\documentclass{amsart}

\usepackage{amssymb}
\usepackage{amsmath}
\usepackage{amsfonts}
\usepackage{geometry}
\usepackage{bbm}
\usepackage{tikz}
\usepackage{mathrsfs}
\usepackage{multirow}
\usetikzlibrary{matrix,arrows,decorations.pathmorphing}

\newcounter{cprop}[section]

\newtheorem{theorem}[cprop]{Theorem}

\theoremstyle{plain}

\newtheorem{lemma}[cprop]{Lemma}
\newtheorem{proposition}[cprop]{Proposition}

\newtheorem{assumption}[cprop]{Assumption}

\numberwithin{equation}{section}

\theoremstyle{definition}
\newtheorem{definition}[cprop]{Definition}

\theoremstyle{remark}
\newtheorem{remark}[cprop]{Remark}

\renewcommand{\P}{\mathbb{P}}

\newcommand{\R}{\mathbb{R}}
\newcommand{\N}{\mathbb{N}}
\newcommand{\Z}{\mathbb{Z}}

\newcommand{\vertiii}[1]{{\left\vert\kern-0.25ex\left\vert\kern-0.25ex\left\vert #1 
    \right\vert\kern-0.25ex\right\vert\kern-0.25ex\right\vert}}

\begin{document}
\title[Oseledets splitting and invariant manifolds on fields of Banach spaces]{Oseledets splitting and invariant manifolds on fields of Banach spaces}

\author{M. Ghani Varzaneh}
\address{Mazyar Ghani Varzaneh\\
Institut f\"ur Mathematik, Technische Universit\"at Berlin, Germany and Department of Mathematical Sciences, Sharif University of Technology, Tehran, Iran}
\email{mazyarghani69@gmail.com}

\author{S. Riedel}
\address{Sebastian Riedel \\
Institut f\"ur Mathematik, Technische Universit\"at Berlin, Germany}
\email{riedel@math.tu-berlin.de}

%\author{(M. Scheutzow)}
%\address{(Michael Scheutzow)}



\begin{abstract}
 We prove a semi-invertible Oseledets theorem for cocycles acting on measurable fields of Banach spaces, i.e. we only assume invertibility of the base, not of the operator. As an application, we prove an invariant manifold theorem for nonlinear cocycles acting on measurable fields of Banach spaces.
\end{abstract}

\maketitle

\section*{Introduction}

The Multiplicative Ergodic Theorem (MET) is a powerful tool with various applications in different fields of mathematics, including analysis, probability theory and geometry, and a cornerstone in smooth ergodic theory. It was first proved by Oseledets \cite{Ose68} for matrix cocycles. Since then, the theorem attracted many researchers to provide new proofs and formulations with increasing generality \cite{Rag79, Rue79, Rue82, Man83, Thi87, Wal93, LL10, DoaPhd09, Blu16, GTQ15}. 

In \cite{GRS19}, the authors gave a proof for an MET for compact cocycles acting on \emph{measurable fields of Banach spaces}. Let us quickly recall the setting here: If $(\Omega,\mathcal{F},\P)$ denotes a probability space, we call a family of Banach spaces $\{E_{\omega} \}_{\omega \in \Omega}$ a measurable field if there exists a linear subspace $\Delta$ of all sections $\Pi_{\omega \in \Omega} E_{\omega}$ and a countable space $\Delta_0 \subset \Delta$ such that $\{g(\omega)\, :\, g \in \Delta_0\}$ is dense in $E_{\omega}$ for every $\omega \in \Omega$ and $\omega \mapsto \| g(\omega) \|_{E_{\omega}}$ is measurable for every $g \in \Delta$. Note that this definition implies that every Banach space $E_{\omega}$ is separable. On the other hand, every separable Banach space defines a field of Banach spaces by simply setting $E_{\omega} = E$. This structure is similar to a measurable version of a Banach bundle with base $\Omega$ and total space $\Pi_{\omega \in \Omega} E_{\omega}$ in which every space $E_{\omega}$ is a fiber. However, the fundamental difference is that we do \emph{not} put any measurable (or topological) structure on the bundle $\Pi_{\omega \in \Omega} E_{\omega}$ itself! In fact, the existence of the set $\Delta$ is a substitute for the measurable structure and will help to prove measurability for functionals defined on $\Pi_{\omega \in \Omega} E_{\omega}$ as we will see many times in this work. If $(\Omega,\mathcal{F},\P,\theta)$ is a measure preserving dynamical systems, a \emph{cocycle} acting on the field $\{E_{\omega}\}_{\omega \in \Omega}$ consists of a family of maps $\varphi_{\omega} \colon E_{\omega} \to E_{\theta \omega}$. Setting $\varphi^n_{\omega} := \varphi_{\theta^{n-1} \omega} \circ \cdots \circ \varphi_{\omega}$, we furthermore claim that $\omega \mapsto \| \varphi^n_{\omega}(g(\omega)) \|_{E_{\theta^n \omega}}$ is measurable for every $g \in \Delta$ and every $n \in \N$. %If all $\varphi_{\omega}$ are continuous (linear, compact) linear, we speak of a continuous (linear, compact) cocycle. 

There are numerous examples in which it is natural to study cocycles on random spaces. In \cite{GRS19}, our motivation was to study dynamical properties of singular stochastic delay differential equations in which the spaces $E_{\omega}$ are (essentially) spaces of controlled Brownian paths known in rough paths theory \cite{FH14}. In the finite dimensional case, linearizing a $C^1$-cocycle on a manifold yields a linear cocycle acting on the tangent bundle \cite[Chapter 4.2]{Arn98}. In the context of stochastic partial differential equations (SPDE), cocycles on random metric spaces were studied, for instance, when uniqueness of the equation is unknown and one has to work with a measurable selection instead, cf. \cite{FS96} in the case of the 3D stochastic Navier-Stokes equation. Other examples in the situation of SPDE can be found in \cite{CKS04, CRS07}. In the deterministic case, a similar structure appears when studying the flow on time-dependent domains \cite{Lio62}. More recently, scales of time-dependent Banach spaces where introduced to study dynamical properties of non-autonomous PDEs in \cite{DPDT11,CPT13}.

 We will now restate the MET \cite[Theorem 4.17]{GRS19} in a slightly simplified version.
 
 \begin{theorem}\label{thm:MET_Banach_fields}
Let $(\Omega,\mathcal{F},\mathbb{P},\theta)$ be an ergodic measurable metric dynamical system and $\varphi$ be a compact linear cocycle acting on a measurable field of Banach spaces $\{E_{\omega}\}_{\omega \in \Omega}$. For $\mu \in \R \cup \{-\infty\}$ and $\omega \in {\Omega}$, define
\begin{align*}
 F_{\mu}(\omega) := \big{\lbrace} x\in E_{\omega}\, :\, \limsup_{n\rightarrow\infty} \frac{1}{n} \log \Vert\varphi^n_{\omega}(x) \Vert \leqslant \mu \big{\rbrace}. 
\end{align*}
Assume that
\begin{align*}
 \log^+ \Vert\varphi_{\omega} \Vert \in L^{1}(\Omega).
\end{align*}
% $\varphi(1,\omega):\pi^{-1}({\omega})\rightarrow\pi^{-1}{(\theta\omega)} $ is compact for every $ \omega\in\Omega $ also assume  $ \theta :\Omega\rightarrow\Omega$ is an ergodic random dynamical system and $ \log\big{(}\Vert\varphi(1,\omega)\Vert\big{)}\in L^{1}(\Omega) $
Then there is a measurable forward invariant set $\tilde{\Omega} \subset \Omega$ of full measure and a decreasing sequence $\{\mu_i\}_{i \geq 1}$, $\mu_i \in [-\infty, \infty)$ with the properties that $\lim_{n \to \infty} \mu_n = - \infty$ and either $\mu_i > \mu_{i+1}$ or $\mu_i = \mu_{i+1} = -\infty$ such that for every $\omega \in \tilde{\Omega}$,
\begin{align}
x\in F_{\mu_{i}}(\omega)\setminus F_{\mu_{i+1}}(\omega) \quad \text{if and only if}\quad \lim_{n\rightarrow\infty}\frac{1}{n}\log \Vert\varphi^n_{\omega}(x) \| = \mu_{i}.
\end{align}
Moreover, there are numbers $m_1, m_2, \ldots$ such that $\operatorname{codim} F_{\mu_j}(\omega) = m_1 + \ldots + m_{j-1}$ for every $\omega \in \tilde{\Omega}$.
\end{theorem}

Let us mention here that, motivated by our example of a stochastic delay equation, we proved this theorem for compact cocycles only, but it should be straightforward to generalize it to the quasi-compact case as in  \cite{GTQ15}. Consequently, we believe that all our results in this work will hold for quasi-compact cocycles, too.

The numbers $\{\mu_i\}$ are the \emph{Lyapunov exponents}, the subspaces $F_{\mu}(\omega)$ are sometimes called \emph{slow-growing subspaces} and the resulting filtration
\begin{align*}
 E_{\omega} = F_{\mu_1}(\omega) \supset F_{\mu_2}(\omega) \supset \ldots  
\end{align*}
is called \emph{Oseledets filtration}. Is is easily seen that the slow-growing spaces are \emph{equivariant}, meaning that $\varphi_{\omega}(F_{\mu_i}(\omega)) \subset F_{\mu_i}(\theta \omega)$. In the proof of this theorem, no invertibility of $\theta$ or $\varphi$ is assumed, in which case a filtration of slow-growing subspaces is the best one can hope for. However, things change when we assume that the base $\theta$ is invertible. In this case, it is possible to deduce a \emph{splitting} of the spaces $E_{\omega}$ consisting of \emph{fast-growing subspaces} which are invariant under $\varphi$. Such a splitting is called \emph{Oseledets splitting}, and the corresponding theorem is called \emph{semi-invertible MET}. Let us emphasize that we only need to assume invertibility of the base $\theta$ and no invertibility of the cocyle $\varphi$. In the context of SPDE or stochastic delay equations, these assumptions are quite natural: $\theta$ usually denotes the shift of a random trajectory (which can be shifted forward and backward in time) and the cocycle denotes the solution map, which is not injective if the equation can be solved forward in time only.

Our first main result is a semi-invertible MET on a measurable field of Banach spaces. We state a simplified version here, the full statement can be found in Theorem \ref{BT} below.

\begin{theorem}\label{theorem:main_1_intro}
  In addition to the assumptions made in Theorem \ref{thm:MET_Banach_fields}, assume that $\theta$ is invertible with measurable inverse $\sigma := \theta^{-1}$ and that Assumption \ref{MES} holds. Then there is a $\theta$-invariant set $\tilde{\Omega}$ of full measure such that for every $i \geq 1$ with $\mu_i > \mu_{i+1}$ and $\omega \in \tilde{\Omega}$, there is an $m_i$-dimensional subspace $H^i_\omega$ with the following properties:
\begin{itemize}
 \item[(i)] (Invariance)\ \ $\phi_{\omega}^k(H^i_{\omega}) = H^i_{\theta^k \omega}$ for every $k \geq 0$.
 \item[(ii)] (Splitting)\ \ $H_{\omega}^i \oplus F_{\mu_{i+1}}(\omega) = F_{\mu_i}(\omega)$. In particular,
 \begin{align*}
  E_{\omega} = H^1_{\omega} \oplus \cdots \oplus H^i_{\omega} \oplus   F_{\mu_{i+1}}(\omega).
 \end{align*}
 \item[(iii)] ('Fast' growing subspace)\ \ For each $ h_{\omega}\in H^{i}_{\omega} $,
 \begin{align*}
  \lim_{n\rightarrow\infty}\frac{1}{n}\log\Vert \phi^{n}_{\omega}(h_{\omega})\Vert = \mu_{j}
 \end{align*}
    and
\begin{align*}
\lim_{n\rightarrow\infty}\frac{1}{n}\log\Vert (\phi^{n}_{\sigma ^{n}\omega})^{-1}(h_{\omega})\Vert =-\mu_{j}.
\end{align*}

\end{itemize}

Moreover, the spaces are uniquely determined by properties (i), (ii) and (iii).

\end{theorem}

Clearly, the Oseledets splitting provides much more information about the cocycle than the filtration. 

Let us discuss some important preceeding results. In the finite dimensional case, an MET for cocycles acting on measurable bundles can be found in the monograph \cite[4.2.6 Theorem]{Arn98} by L.~Arnold. In \cite{Man83}, Ma\~{n}\'{e} proved an MET with Oseledets splitting on a Banach bundle, assuming a topological structure on $\Omega$ and continuity of the map $\omega \mapsto \varphi_{\omega}$. He also assumed injectivity of $\varphi$. Besides these results, we are not aware of any METs for cocycles acting on a bundle-type structure. Lian and Lu \cite{LL10} prove an MET for cocycles acting on a fixed Banach space, assuming only a measurable structure on $\Omega$, but injectivity of the cocycle. This assumption was later removed by Doan in \cite{DoaPhd09} without giving an Oseledets splitting, however. In \cite{GTQ14}, Gonz\'{a}lez-Tokman and Quas used this result as a ``black-box'' and proved that an Oseledets splitting holds in this case, too.

Let us mention that our result is not only the first which provides a splitting on a bundle structure of Banach spaces without using a topological structure on $\Omega$, it also weakens the measurability assumption on $\varphi$ significantly in case we are dealing with a single Banach space $E$ only. In fact, the standard measurability assumption, for instance in \cite{GTQ15}, is \emph{strong measurability} of $\varphi$, meaning that for fixed $x \in E$, the map
\begin{align}\label{eqn:strong_measurability}
 \Omega \ni \omega \mapsto \varphi_{\omega}(x) \in E
\end{align}
should be measurable. In contrast, our assumption means that the maps
\begin{align*}
 \Omega \ni \omega \mapsto \| \varphi^{k + n}_{\omega}(x) - \varphi^{k}_{\theta^n \omega}(\tilde{x}) \|_E \in \R
\end{align*}
should be measurable for every $n,k \in \N_0$ and $x, \tilde{x} \in S$ where $S$ is a countable and dense subset of $E$. This assumption is clearly implied by \eqref{eqn:strong_measurability}. 

The proof of Theorem \ref{theorem:main_1_intro} pushes forward the volume growth-approach advocated by Blumenthal \cite{Blu16} and Gonz\'{a}lez-Tokman, Quas \cite{GTQ15} which provides a clear growth interpretation of the Lyapunov exponents. In a way, our result complements these two works in case of a single Banach space $E$. In particular, we are not imposing any further assumptions on $E$ like reflexivity or separability of the dual as in \cite{GTQ15}.

A typical application for an MET is the construction of stable and unstable manifolds, cf. \cite{Rue79, Rue82, Man83}. Here, the existence of the Oseledets splitting is crucial. Our second main contribution is an invariant manifold theorem for nonlinear cocycles acting on fields of Banach spaces. We state an informal version here, the precise statements are formulated in Theorem \ref{stable manifold} and Theorem \ref{unstable manifold}.

\begin{theorem}
  Let $\varphi$ be a nonlinear, differentiable cocycle acting on a measurable field of Banach spaces $\{E_{\omega} \}_{\omega \in \Omega}$. Assume that $Y_{\omega}$ is a random fixed point of $\varphi$, in particular $\varphi_{\omega}(Y_{\omega}) = Y_{\theta \omega}$. Then, under the same measurability and integrability assumptions as in Theorem \ref{theorem:main_1_intro}, the linearized cocycle $D_{Y_{\omega}} \varphi_{\omega}$ has a Lyapunov spectrum $\{\mu_n\}_{n \geq 1}$. Under further assumptions on $\varphi$ and $Y$, there is a $\theta$-invariant set $\tilde{\Omega}$ of full measure, closed subspaces $S_{\omega}$ and $U_{\omega}$ of $E_{\omega}$ 
%   \begin{align*}
%    E_{\omega} = U_{\omega} \oplus S_{\omega}
%   \end{align*}
  and immersed submanifolds $S_{loc}(\omega)$ and $U_{loc}(\omega)$ of $E_{\omega}$ such that for every $\omega \in \tilde{\Omega}$,
  \begin{align*}
   T_{Y(\omega)} S_{loc}(\omega) = S_{\omega} \qquad \text{and} \qquad T_{Y(\omega)} U_{loc}(\omega) = U_{\omega}
  \end{align*}
  and the properties that for every $ Z_{\omega}\in S_{loc}(\omega) $,
\begin{align*}
    \limsup_{n\rightarrow\infty}\frac{1}{n}\log\Vert\varphi^{n}_{\omega}(Z_{\omega})-Y_{\theta^{n}\omega}\Vert\leqslant  \mu_{j_{0}} < 0
\end{align*}
   and for every $ Z_{\omega}\in U_{loc}(\omega) $ one has $\varphi^n_{\sigma^n \omega}(Z_{\sigma^n \omega}) = Z_{\omega}$ and
\begin{align*}
\limsup_{n\rightarrow\infty}\frac{1}{n}\log\Vert Z_{\sigma^{n}\omega}-Y_{\sigma^{n}\omega}\Vert\leqslant -\mu_{k_{0}} < 0.
\end{align*}
Here we have set $\mu_{j_0} = \max \{ \mu_j \, :\ \mu_j < 0 \}$ and $\mu_{k_0} = \min \{ \mu_k \, :\, \mu_k > 0 \}$. In the hyperbolic case, i.e. if all Lyapunov exponents are non-zero, the submanifolds $S^{\upsilon}_{loc}(\omega) $ and $U^{\upsilon}_{loc}(\omega) $ are {transversal}, i.e.
  \begin{align*}
   E_{\omega} = T_{Y_{\omega}} U^{\upsilon}_{loc}(\omega) \oplus T_{Y_{\omega}} S^{\upsilon}_{loc}(\omega).
  \end{align*}

\end{theorem}

The structure of the paper is as follows. In Section \ref{sec:semi_MET}, we prove a semi-invertible MET for cocycles acting on measurable fields of Banach spaces. This result is applied in Section \ref{sec_inf_mfd} to deduce the existence of local stable and unstable manifolds for nonlinear cocycles.

\subsection*{Notation}

\begin{itemize}
 \item For Banach spaces $(X, \|\cdot\|_X)$ and $(Y, \|\cdot\|_Y)$, $ L(X,Y) $ denotes the space of bounded linear functions from $ X $ to $Y $ equipped with usual operator norm. We will often not explicitly write a subindex for Banach space norms and use the symbol $\| \cdot \|$ instead. \emph{Differentiability} of a function $f  \colon X \to Y$ will always mean Fr\'echet-differentiability. A \emph{$C^m$ function} denotes an $m$-times Fr\'echet-differentiable function. If $A, B \subseteq X$, we denote by $d(A,B) := \inf_{a \in A, b \in B} \|a - b\|$ the distance between two sets $A$ and $B$. We also set $d(x,B) := d(B,x) := d(\{x\},B)$ for $x \in X$, $B \subseteq X$.
 
 \item Let $ X ,Y $ be  Banach spaces. For $ x_{1},...,x_{k}\in X $, set % the  of $ (x_{1},x_{2},...,x_{k}) $ by
\begin{align}\label{VOL}
  \operatorname{Vol}(x_{1},x_{2},...,x_{k}) := \Vert x_{1}\Vert\prod_{i=2}^{k}d(x_{i},\langle x_{j} \rangle_{1\leqslant j<i}).
 \end{align}
 For a given bounded linear function $ T : X\rightarrow Y $ and $k \geq 1$, we define
\begin{align*}
D_{k}(T):=\sup_{\Vert x_{i}\Vert =1 ; i=1,...,k} \operatorname{Vol} \big{(}T(x_{1}),T(x_{2}),...,T(x_{k})\big{)}.
\end{align*}
 
  \item Let $E$ be a vector space. If we can write $E$ as a direct sum $E = F \oplus H$ of vector spaces, we have an \emph{algebraic splitting}. We also say that \emph{$F$ is a complement of $H$} and vice versa. The projection operator $\Pi_{F \Vert H}(e) = f$ with $e = f + h$, $f \in F$, $h \in H$, is called the \emph{projection operator onto $F$ parallel to $H$}. If $E$ is a normed space and $\Pi_{F \Vert H}$ is bounded linear, i.e.
\begin{align*}
 \| \Pi_{F \Vert H} \| = \sup_{f \in F, h \in H, f + h \neq 0} \frac{\| f \|}{\|f + h\|} < \infty,
\end{align*}
we call $E = F \oplus H$ a \emph{topological splitting}. For normed spaces, a \emph{splitting} will always mean a topological splitting.

 \item Let $(\Omega,\mathcal{F})$ be a measurable space. We call a family of Banach spaces $\{E_{\omega}\}_{\omega \in \Omega}$ a \emph{measurable field of Banach spaces} if there is a set of sections
 \begin{align*}
  \Delta \subset \prod_{\omega \in \Omega} E_{\omega}
 \end{align*}
 with the following properties:
 \begin{itemize}
  \item[(i)] $\Delta$ is a linear subspace of $\prod_{\omega \in \Omega} E_{\omega}$.
  \item[(ii)] There is a countable subset $\Delta_0 \subset \Delta$ such that for every $\omega \in \Omega$, the set $\{g(\omega)\, :\, g \in \Delta_0\}$ is dense in $E_{\omega}$.
  \item[(iii)] For every $g \in \Delta$, the map $\omega \mapsto \| g(\omega) \|_{E_{\omega}}$ is measurable.
 \end{itemize}
 
 \item Let $(\Omega,\mathcal{F})$ be a measurable space. If there exists a measurable map $\theta \colon \Omega \to \Omega$, $\omega \mapsto \theta \omega$, with a measurable inverse $\theta^{-1}$, we call $(\Omega,\mathcal{F}, \theta)$ a \emph{measurable dynamical system}. We will use the notation $\theta^n \omega$ for $n$-times applying $\theta$ to an element $\omega \in \Omega$. We also set $\theta^0 := \operatorname{Id}_{\Omega}$ and $\theta^{-n} := (\theta^n)^{-1}$. If $\P$ is a probability measure on $(\Omega,\mathcal{F})$ that is invariant under $\theta$, i.e. $\P(\theta^{-1} A) = \P(A) = \P( \theta A)$ for every $A \in \mathcal{F}$, we call the tuple $\big(\Omega, \mathcal{F},\P,\theta\big)$ a \emph{measure-preserving dynamical system}. The system is called \emph{ergodic} if every $\theta$-invariant set has probability $0$ or $1$.

 \item Let $(\Omega,\mathcal{F},\P,\theta)$ be a measure-preserving dynamical system and $(\{E_{\omega}\}_{\omega \in \Omega},\Delta)$ a measurable field of Banach spaces. A \emph{continuous cocycle on $\{E_{\omega}\}_{\omega \in \Omega}$} consists of a family of continuous maps
 \begin{align}\label{eqn:def_cocycle}
  \varphi_{\omega} \colon E_{\omega} \to E_{\theta \omega}.
 \end{align}
 If $\varphi$ is a continuous cocycle, we define $\varphi^n_{\omega} \colon E_{\omega} \to E_{\theta^n \omega}$ as 
 \begin{align*}
  \varphi^n_{\omega} := \varphi_{\theta^{n-1}\omega} \circ \cdots \circ \varphi_{\omega}.
 \end{align*}
 We also set $\varphi^0_{\omega} := \operatorname{Id}_{E_{\omega}}$. We say that \emph{$\varphi$ acts on $\{E_{\omega}\}_{\omega \in \Omega}$} if the maps
 \begin{align*}
  \omega \mapsto \| \varphi(n,\omega,g(\omega)) \|_{E_{\theta^n \omega}}, \quad n \in \N
 \end{align*}
 are measurable for every $g \in \Delta$. In this case, we will speak of a \emph{continuous random dynamical system on a field of Banach spaces}. If the map \eqref{eqn:def_cocycle} is bounded linear/compact, we call $\varphi$ a bounded linear/compact cocycle.
 
\end{itemize}

\section{Semi-invertible MET on fields of Banach spaces}\label{sec:semi_MET}

In this section, $(\Omega,\mathcal{F},\P,\theta)$ will denote an ergodic measure-preserving dynamical system and we set $\sigma := \theta^{-1}$. Let $(\{E_{\omega})_{\omega \in \Omega},\Delta,\Delta_0)$ be a measurable field of Banach space and let $\psi_{\omega} \colon E_{\omega} \to E_{\theta \omega}$ be a compact linear cocycle acting on it. In the sequel, we will furthermore assume that the following assumption is satisfied:

\begin{assumption}\label{MES}
For each $ g ,\tilde{g}\in \Delta$ and $n, k \geq 0$,
\begin{align*}
\omega\rightarrow\Vert\psi_{\theta^n \omega}^k ( \psi^{n}_{\omega}(g_{\omega})-\tilde{g}_{\theta^{n}\omega} )\Vert_{E_{\theta^{n + k} \omega}}
\end{align*}
is measurable.
\end{assumption}

We will always assume that 
\begin{align*}
 \log^+ \| \psi_{\omega} \| \in L^1(\Omega).
\end{align*}
Under this condition, the Multiplicative Ergodic Theorem \cite[Theorem 4.17]{GRS19} applies and yields the existence of Lyapunov exponents $\{\mu_1 > \mu_2 > \ldots \} \subset [-\infty,\infty)$ on a $\theta$-invariant set of full measure $\tilde{\Omega} \subset \Omega$. More precisely, there are numbers $\Lambda_k \in [-\infty,\infty)$ such that
\begin{align*}
 \Lambda_k = \lim_{n\rightarrow\infty}\frac{1}{n}\log D_{k}\big{(}\psi^{n}_{\omega}\big{)}, \quad k \geq 1
\end{align*}
for every $\omega \in \tilde{\Omega}$. Setting $\lambda_k = \Lambda_k - \Lambda_{k-1}$, the sequence $(\mu_k)$ is the subsequence of $(\lambda_k)$ defined by removing all multiple elements. For any $\mu \in [-\infty,\infty)$, we define the closed subspace
\begin{align*}
 F_{\mu}(\omega) = \left\{ \xi \in E_{\omega} \, |\, \limsup_{n \to \infty} \frac{1}{n} \log \|\psi^n_{\omega}(\xi) \| \leq \mu \right\}.
\end{align*}
Note that $\psi$ is invariant on these spaces in the sense that
\begin{align*}
 \psi^{n}_{\omega} \vert_{F_{\mu}(\omega)} \colon {F_{\mu}(\omega)} \to {F_{\mu}(\theta^n \omega)}.
\end{align*}
We also saw in \cite[Theorem 4.17]{GRS19} that there are numbers $m_i \in \N$ such that $m_i = \operatorname{dim} \left( F_{\mu_i}(\omega) / F_{\mu_{i+1}}(\omega)\right)$ for every $\omega \in \tilde{\Omega}$.

If not otherwise stated, $\tilde{\Omega} \subset \Omega$ will always denote a $\theta$-invariant set of full measure. Note that we can always assume w.l.o.g. that a given set of full measure $\Omega_0 \subset \Omega$ is $\theta$-invariant, otherwise we can consider
\begin{align*}
 \bigcap_{k \in \Z} \theta^k(\Omega_0)
\end{align*}
instead.

Next, we collect some basic Lemmas. Recall the definition of $\operatorname{Vol}$ and $D_k$.

\begin{lemma}\label{dual}
Let $ X, Y $ be Banach spaces and $ T:X\rightarrow Y $ a linear operator. For $k\in\mathbb{N}$, there exist positive constants $ c_{k}, C_{k} $  depending only on $k$ such that
\begin{align}
c_{k}D_{k}(T)\leqslant D_{k}(T^{*})\leqslant C_{k} D_{k}(T)
\end{align} 
where by $ T^{*}:Y^{*}\rightarrow X^{*} $ we mean the dual map of $ T $.
 \end{lemma}

\begin{proof}
\cite[Lemma 3]{GTQ15}.
\end{proof}

\begin{lemma}\label{Con_VOl}
For a Banach space $ X $ and $ k\geqslant 1 $, the map
\begin{align}
\begin{split}
\operatorname{Vol} : X^{k} &\longrightarrow\mathbb{R}\\ 
(x_{1},x_{2},...,x_{k}) &\mapsto \Vert x_{1}\Vert\prod_{i=2}^{k}d(x_{i},\langle x_{j} \rangle_{1\leqslant j<i})
 \end{split}
 \end{align}
 is continuous.
\end{lemma}

\begin{proof}
\cite[Lemma 4.2]{LL10}.
\end{proof}

\begin{lemma}\label{mea_distance}
For every $ g\in\Delta $ and $ j\geqslant 1 $, the map
\begin{align*}
\omega \mapsto d\big{(}g(\omega) , F_{\mu_{j}}(\omega))\big{)}
\end{align*}
is measurable.
\end{lemma}

\begin{proof}
 As in the proof to \cite[Lemma 4.3]{GRS19}.
\end{proof}

For a Banach space $X$ and a closed subspace $U \subset X$, the quotient space $X / U$ is again a Banach space with norm
\begin{align*}
 \| [x] \|_{X/U} = \inf_{u \in U} \| x-u \|.
\end{align*}
For an element $x \in E_{\omega}$, we denote by $[x]_{\mu}$ its equivalence class in the quotient space $E_{\omega} / F_{\mu}(\omega)$. From the invariance property of $\psi$, the map 
\begin{align*}
 [\psi^{n}_{\omega}]_{\mu_{j+1}}: \frac{F_{\mu_{j}}(\omega)}{F_{\mu_{j+1}}(\omega)} &\longrightarrow \frac{F_{\mu_{j}}(\theta^{n}\omega)}{F_{\mu_{j+1}}(\theta^{n}\omega)}, \quad [\psi^{n}_{\omega}]_{\mu_{j+1}}([x]) := [ \psi^n_{\omega}(x)]_{\mu_{j+1}} \\
\end{align*}
is well-defined for every $j \geq 1$ and $n \in \N$. Note also that $[\psi^{n}_{\omega}]_{\mu_{j+1}}$ is bijective for $\omega \in \tilde{\Omega}$. Indeed, injectivity is straightforward and surjectivity follows from the fact that $F_{\mu_{j}}(\omega)/F_{\mu_{j+1}}(\omega)$ and $F_{\mu_{j}}(\theta^n\omega)/F_{\mu_{j+1}}(\theta^n\omega)$ are finite-dimensional with the same dimension $m_i$.

\begin{lemma}\label{mea_Vol}
 For $ j,m, n\in\mathbb{N}$, the maps
 \begin{align*}
 f_{1}(\omega):=  D_{m}(\psi^{n}_{\omega} \mid_{F_{\mu_{j}}(\omega)}) \  \ \ \text{and} \ \ \ \  f_{2}(\omega):=D_{m}{(}{[}\psi^{n}_{\omega}{]}_{\mu_{j+1}}{)}
 \end{align*}
  are measurable.
%  Where 
% \begin{align*}
% [\psi^{n}_{\omega}]_{j+1}: \frac{F_{\mu_{j}}(\omega)}{F_{\mu_{j+1}}(\omega)}\longrightarrow \frac{F_{\mu_{j}}(\theta^{n}\omega)}{F_{\mu_{j+1}}(\theta^{n}\omega)}
% \end{align*}
\end{lemma}

\begin{proof}
It is not hard to see that
\begin{align}
&f_{1}(\omega)=\lim _{l \rightarrow\infty}\liminf_{k \rightarrow\infty}\bigg{[}\sup_{\lbrace\xi^{t}_{\omega}\rbrace_{1\leqslant t\leqslant m} \subset B_{\omega}^{l,k}(\mu_{j})} \operatorname{Vol}\big{(}\psi^{n}_{\omega}(\xi^{1}_{\omega}),...,\psi^{n}_{\omega}(\xi^{m}_{\omega})\big{)}\bigg{]}
\end{align}
where
\begin{align*}
B_{\omega}^{l,k}(\mu_{j})=\bigg{\lbrace} \xi \in E_{\omega}:\ \| \xi\| = 1, \ &\Vert\psi^k_{\omega}(\xi) \Vert< \exp \big{(}  k(\mu_{j} + \frac{1}{l}) \big{)}\ \text{and}\ \\ &d\big{(}\xi, F_{\mu_{i}}(\omega)\big{)}<\exp\big{(}k(\mu_{j}-\mu_{i-1})\big{)}, 1\leqslant i<j  \bigg{\rbrace},
\end{align*}
cf. the proof of \cite[Lemma 4.3]{GRS19}. Let $\lbrace g_{t}\rbrace_{1\leqslant t\leqslant m}\subset \Delta_{0} $ and $ C({g_{t}}) := \lbrace \omega\ :\ {g}_{t}(\omega)\in B^{l,k}_{\omega}(\mu_{j})\rbrace $. As a consequence of Lemma \ref{mea_distance}, these sets are measurable and we have 
\begin{align*}
\sup_{\lbrace\xi_{\omega}^{t}\rbrace_{1\leqslant t\leqslant m} \subset B_{\omega}^{l,k}(\mu_{j})}& \operatorname{Vol}\big{(}\psi^{n}_{\omega}(\xi^{1}_{\omega}),...,\psi^{n}_{\omega}(\xi^{m}_{\omega})\big{)} = \\&\sup_{\lbrace g_{t}\rbrace_{1\leqslant t\leqslant m}\subset \Delta_0 } \operatorname{Vol}\bigg{(}\psi^{n}_{\omega}\big{(}\frac{g_{1}(\omega)}{\Vert g_{1}(\omega)\Vert} \big{)},...,\psi^{n}_{\omega}\big{(}\frac{g_{m}(\omega)}{\Vert g_{m}(\omega)\Vert} \big{)}\bigg{)}\, \prod_{1\leqslant t\leqslant m}\chi_{C(g_{t})}(\omega)
\end{align*}
which implies measurability of $f_1$. For $f_2$, note first that
\begin{align*}
f_{2}(\omega)=\lim _{l \rightarrow\infty}\liminf_{k \rightarrow\infty}
\bigg{[}\sup_{\lbrace\xi^{t}_{\omega}\rbrace_{1\leqslant t\leqslant m} \subset B_{\omega}^{l,k}(\mu_{j})}& \frac{\operatorname{Vol}\big{(}[\psi^{n}_{\omega}(\xi^{1}_{\omega}){]}_{\mu_{i+1}},...,{[}\psi^{n}_{\omega}(\xi^{m}_{\omega}){]}_{\mu_{i+1}}\big{)}}{\prod_{1\leqslant t\leqslant m}\Vert[\xi^{t}_{\omega}]_{\mu_{j+1}}\Vert}\bigg{]}
\end{align*}
where we set $ \frac{0}{0}:=0 $. Again as before
\begin{align*}
\sup_{\lbrace\xi^{t}_{\omega}\rbrace_{1\leqslant t\leqslant m} \subset B_{\omega}^{l,k}(\mu_{j})}& \frac{\operatorname{Vol}\big{(}[\psi^{n}_{\omega}(\xi^{1}_{\omega}){]}_{\mu_{i+1}},...,{[}\psi^{n}_{\omega}(\xi^{m}_{\omega}){]}_{\mu_{j+1}}\big{)}}{\prod_{1\leqslant t\leqslant m}\Vert[\xi^{t}_{\omega}]_{\mu_{j+1}}\Vert} = \\ &\sup_{\lbrace g_{t}\rbrace_{1\leqslant t\leqslant m}\subset \Delta_{0} }\frac{\operatorname{Vol}\big{(}{[}\psi^{n}_{\omega}{(}g_{1}(\omega){)}\big{]}_{\mu_{i+1}},...,{[}\psi^{n}_{\omega}(g_{k}(\omega))]_{\mu_{i+1}}\big{)}}{\prod_{1\leqslant t\leqslant m}d\big{(}g_{t}(\omega),F_{\mu_{i+1}}(\omega)\big{)}} \, \prod_{1\leqslant t\leqslant m}\chi_{C(g_{t})}(\omega).
\end{align*}
It remains to show that for $ g\in \Delta $,\ {$  d\big{(}\psi^{n}_{\omega}\big{(}{g(\omega)}\big{)}, F_{\mu_{j+1}}(\theta^{n}\omega)\big{)}$} is measurable, which can be achieved using Assumption \ref{MES} with a proof similar to Lemma \ref{mea_distance}.
\end{proof}

\begin{lemma}\label{fra_L}
 For every $i \geq 0$, there is a constant $M_i > 0$ such that
\begin{align*}
\Vert [\psi^{1}_{\omega}]_{\mu_{i+1}}\Vert< M_{i}\Vert\psi^{1}_{\omega}\Vert
\end{align*}
for every $\omega \in \tilde{\Omega}$.

\end{lemma}

\begin{proof}
 Since $\operatorname{dim}[\frac{F_{\mu_{i}}(\omega)}{F_{\mu_{i+1}}(\omega)}] = m_{i} $, we can choose $ H_{\omega}\subset F_{\mu_{i}}(\omega)$ such that 
 \begin{align}\label{CXCC}
H_{\omega}\oplus F_{\mu_{i+1}}(\omega)=F_{\mu_{i}}(\omega)\ \ \  \text{and} \ \ \ \ \Vert\Pi_{H_{\omega}||F_{\mu_{i+1}}(\omega)}\Vert \leq \sqrt{m_{i}}+2 =:  M_i,
\end{align}
cf. \cite[III.B.11]{Woj91}. Let $ \xi_{\omega}\in F_{\mu_{i}}(\omega)\setminus F_{\mu_{i+1}}(\omega) $ with corresponding decomposition $ \xi_{\omega} = h_{\omega} + f_{\omega} \in H_{\omega}\oplus F_{\mu_{i+1}}(\omega)$. From \eqref{CXCC}, we know that $ \frac{\Vert h_{\omega}\Vert}{\Vert [\xi_{\omega}]_{\mu_{i+1}}\Vert}\leqslant M_{i} $ and consequently
\begin{align*}
 \frac{\| [ \psi^1_{\omega}(\xi_{\omega}) ]_{\mu_{i+1}} \|}{ \| [ \xi_{\omega}]_{\mu_{i+1}} \|} \leq M_i \frac{\| [ \psi^1_{\omega}(h_{\omega}) ]_{\mu_{i+1}} \|}{ \| h_{\omega} \|} \leq M_i \frac{\| \psi^1_{\omega}(h_{\omega})  \|}{ \| h_{\omega} \|} \leq M_i \Vert\psi^{1}_{\omega}\Vert.
\end{align*}
The claim follows.

\end{proof}

\begin{lemma}\label{lemma:reverse_sub}
 Assume that $ \lbrace f_{n}(\omega)\rbrace_{n\geqslant 1} $ is a subadditive sequence with respect to $ \theta $ and set $ g_{n}(\omega):=f_{n}(\sigma^{n}\omega)$. Assume $ f^+_{1}(\omega) \in L^{1}(\Omega)$. Then there is a $\theta$-invariant set $\tilde{\Omega} \in \mathcal{F}$ with full measure such that for every $\omega \in \tilde{\Omega}$,
    \begin{align*}
    \lim_{n\rightarrow\infty}\frac{1}{n}f_{n}(\omega)=\lim_{n\rightarrow\infty}\frac{1}{n}g_{n}(\omega)\in [-\infty ,\infty)
    \end{align*}
    where the limit does not depend on $\omega$.
\end{lemma}

\begin{proof}
 We can easily check that $ \lbrace g_{n}(\omega)\rbrace_{n\geqslant 1} $ is a subadditive sequence with respect to $ \sigma $. Since $ f_{n}(\omega) $ and $ g_{n}(\omega) $ have same law, the result follows from Kingman's Subadditive Ergodic Theorem.
\end{proof}

As a consequence, we obtain the following:

\begin{lemma}\label{lemma:forw_back}
 There is a $\theta$-invariant set of full measure $\tilde{\Omega} \in \mathcal{F}$ such that
 \begin{align}\label{Du_Re}
&\lim_{n\rightarrow\infty}\frac{1}{n}\log D_{k}\big{(}\psi^{n}_{\omega}\big{)}=\lim_{n\rightarrow\infty}\frac{1}{n}\log D_{k}\big{(}\psi^{n}_{\sigma^{n}\omega}\big{)}=\lim_{n\rightarrow\infty}\frac{1}{n}\log D_{k}\big{(}(\psi^{n}_{\sigma^{n}\omega})^{*}\big{)}=\Lambda_{k}
\end{align}
and 
\begin{align}\label{Du_Re_2}
    \begin{split}
 &\lim_{n\rightarrow\infty}\frac{1}{n}\log D_{k}\big{(}\psi^{n}_{\omega}\mid_{F_{\mu_{i}}(\omega)}\big{)}=\lim_{n\rightarrow\infty}\frac{1}{n}\log D_{k}\big{(}\psi^{n}_{\sigma^{n}\omega}\mid_{F_{\mu_{i}}(\sigma^{n}\omega)}\big{)}\\&=\lim_{n\rightarrow\infty}\frac{1}{n}\log D_{k}\big{(}(\psi^{n}_{\sigma^{n}\omega})^{*}\mid_{\big{(}F_{\mu_{i}}(\sigma^{n}\omega)\big{)}^{*}}\big{]}=\Lambda_{k+\tilde{m}_{i-1}}-\Lambda_{\tilde{m}_{i-1}}
 \end{split}
\end{align}
where $ \tilde{m}_{0}=0$ and $ \tilde{m}_{i}=\sum_{1\leqslant t\leqslant i}m_{j} $ for $i\geqslant 1 $.

\end{lemma}

\begin{proof}
    We already noted that $\lim_{n\rightarrow\infty}\frac{1}{n}\log D_{k}\big{(}\psi^{n}_{\omega}\big{)} = \Lambda_k$. The equality
    \begin{align}\label{eqn:growth_slow_subspace}
        \lim_{n\rightarrow\infty}\frac{1}{n}\log D_{k}\big{(}\psi^{n}_{\omega}\mid_{F_{\mu_{i}}(\omega)}\big{)} = \Lambda_{k+\tilde{m}_{i-1}}-\Lambda_{\tilde{m}_{i-1}}
    \end{align}
    was a partial result in the proof of Theorem \cite[Theorem 4.17]{GRS19}. The remaining inequalities follow by a combination of all Lemmas \ref{dual} - \ref{lemma:reverse_sub}.
\end{proof}

From now on, we will assume that $\tilde{\Omega}$ is the set provided in Lemma \ref{lemma:forw_back}.

\begin{lemma}\label{fra_fra}
    Fix $i \geq 1$ and $\omega \in \tilde{\Omega}$. Let $(\xi_{\sigma^{n}\omega})_n$ be a sequence such that $ \xi_{\sigma^{n}\omega}\in F_{\mu_{i}}(\sigma^{n}\omega)\setminus F_{\mu_{i+1}}(\sigma^{n}\omega) $ and $\Vert[\xi_{\sigma^{n}\omega}]_{\mu_{i+1}}\Vert =1$ for every $n \in \N$. Then
\begin{align}\label{Fra_lim}
\lim_{n\rightarrow\infty}\frac{1}{n}\log\Vert [\psi^{n}_{\sigma^{n}\omega}(\xi_{\sigma^{n}\omega})]_{\mu_{i+1}}\Vert=\mu_{i}
\end{align}
on a $\theta$-invariant set of full measure.
\end{lemma}

\begin{proof}
By applying Lemma \ref{mea_Vol}, Lemma \ref{fra_L} and Lemma \ref{lemma:reverse_sub}, Kingman's Subadditive Ergodic Theorem shows that
\begin{align*}
\lim_{n\rightarrow\infty}\frac{1}{n}\log D_{m}\big{(} \big{[}\psi^{n}_{\omega}\big{]}_{\mu_{i+1}}\big{)}=\lim_{n\rightarrow\infty}\frac{1}{n}\log D_{m}\big{(}\big{[}\psi^{n}_{\sigma^{n}\omega}\big{]}_{\mu_{i+1}}\big{)}
\end{align*}
exist for every $k \geq 1$. Let $ H_{\omega} $ be a complement subspace for $ F_{\mu_{i+1}}(\omega) $ in $ F_{\mu_{i}}(\omega) $. Using a slight generalization of \cite[Lemma 4.4]{GRS19}, we have that
\begin{align*}
\lim_{n\rightarrow\infty}\frac{1}{n}\log\Vert\Pi_{\psi^{n}_{\omega}(H_{\omega})||F_{\mu_{i+1}}(\theta^{n}\omega)}\Vert =0.
\end{align*}
For $ \xi_{\omega}\in F_{\mu_{i}}(\omega)\setminus F_{\mu_{i+1}}(\omega) $, since
\begin{align*}
 \frac{\Vert\psi^{n}_{\omega}(\Pi_{H_{\omega}||F_{\mu_{i+1}}(\omega)}(\xi_{\omega}))\Vert}{\Vert [\psi^{n}_{\omega}(\xi_{\omega})]_{\mu_{i+1}}\Vert}\leqslant  \Vert \Pi_{\psi^{n}_{\omega}(H_{\omega})||F_{\mu_{i+1}}(\theta^{n}\omega)}\Vert
\end{align*}
it follows that
\begin{align}\label{MMNB}
\lim_{n\rightarrow\infty}\frac{1}{n}\log\Vert [\psi^{n}_{\omega}(\xi_{\omega})]_{\mu_{i+1}}\Vert =\mu_{i}.
\end{align}
% From \cite [Proposition 4.15]{GRS19}  for $ m\leqslant m_{i} $ 
% \begin{align*}
% \lim_{n\rightarrow\infty}\frac{1}{n}\log D_{m}\big{(} \big{[}\psi^{n}_{\omega}\big{]}_{\mu_{j+1}}\big{)} =m \mu_{i}.
% \end{align*}
% Otherwise exists $ \xi_{\omega}\in F_{\mu_{i}}(\omega)\setminus F_{\mu_{i+1}}(\omega) $ such that $ \limsup_{n\rightarrow\infty}\frac{1}{n}\log \Vert [\psi^{n}_{\omega}(\xi_{\omega})]_{\mu_{i+1}}\Vert <\mu_{i}$ which is contradiction by \eqref{MMNB}. 
Let
\begin{align*}
k := \max \big{\lbrace} m : \lim_{n\rightarrow\infty}\frac{1}{n}\log D_{m}\big{(} \big{[}\psi^{n}_{\omega}\big{]}_{\mu_{j+1}}\big{)} =m \mu_{i} \big{\rbrace}.
\end{align*}
We claim $ k=m_{i} $. Indeed, otherwise from \cite [Proposition 4.15]{GRS19}, there exists a subspace $ F_{\omega}\subset \frac{F_{\mu_{i}}(\omega)}{F_{\mu_{i+1}}(\omega)} $ with codimension $ k $ such that for every $ \xi_{\omega}\in F_{\omega} $ 
\begin{align*}
\limsup_{n\rightarrow\infty}\frac{1}{n}\log\Vert [\psi^{n}_{\omega}(\xi_{\omega})]_{\mu_{i+1}} \Vert <\mu_{i}.
\end{align*}
Since $\operatorname{dim} [\frac{F_{\mu_{i}}(\omega)}{F_{\mu_{i+1}}(\omega)}]=m_{i}$, we can find a non-zero element in $ F_{\omega}$ which contradicts \eqref{MMNB}. Hence we have shown that
\begin{align*}
\lim_{n\rightarrow\infty}\frac{1}{n}\log D_{m}\big{(} \big{[}\psi^{n}_{\omega}\big{]}_{\mu_{j+1}}\big{)} =m_i \mu_{i}.
\end{align*}
Therefore, for every $n\in \mathbb{N}$, we can find $\lbrace\xi^{j}_{\sigma^{n}\omega}\rbrace_{1\leqslant j\leqslant m_{i}}\subset F_{\mu_{i}}(\sigma^{n}\omega)$ such that $\Vert[\xi^{j}_{\omega}]_{\mu_{i+1}}\Vert=1 $ and 
\begin{align}
\lim_{n\rightarrow\infty}\frac{1}{n}\operatorname{Vol}\big{(}[\psi^{n}_{\sigma^{n}\omega}(\xi^{1}_{\sigma^{n}\omega})]_{\mu_{i+1}}, \ldots, [\psi^{n}_{\sigma^{n}\omega}(\xi^{m_{i}}_{\sigma^{n}\omega})]_{\mu_{i+1}} \big{)} \big{]}=m_{i} \mu_{i}.
\end{align}
Using the definition of $\operatorname{Vol}$, it follows that for every $ 2\leqslant t\leqslant m_{i} $,
\begin{align}\label{d_d}
\lim_{n\rightarrow\infty}\frac{1}{n}\log d\big{(}[\psi^{n}_{\sigma^{n}\omega}(\xi^{t}_{\omega})]_{\mu_{i+1}},\langle [\psi^{n}_{\sigma^{n}\omega}(\xi^{j}_{\sigma^{n}\omega})]_{\mu_{i+1}}\rangle_{1\leqslant j\leqslant t-1}\big{)}=\mu_{i}
\end{align}
  We have $ \xi_{\sigma^{n}\omega}=\sum_{1\leqslant j\leqslant m_{i}}\alpha_{j}\xi^{j}_{\sigma^{n}\omega}$ mod $ F_{\mu_{i+1}}(\sigma^{n}\omega)$. In the proof of \cite[Lemma 4.7]{GRS19}, we already saw that the the $ \operatorname{Vol} $-function is symmetric up to a constant. By our assumption on $ \xi_{\sigma^{n}\omega} $, we can therefore assume that $ \alpha_{m_{i}}\geqslant \frac{1}{ m_{i}}$. Finally from \eqref{d_d} 
\begin{align*}
\lim_{n\rightarrow\infty}\frac{1}{n}\log \Vert [\psi^{n}_{\sigma^{n}\omega}(\xi_{\sigma^{n}\omega})]_{\mu_{i+1}}\Vert=\lim_{n\rightarrow\infty}\frac{1}{n}\big{[}d\big{(}[\psi(\xi^{m_{i}}_{\sigma^{n}\omega})]_{\mu_{i+1}},\langle [\psi^{n}_{\sigma^{n}\omega}(\xi^{j}_{\sigma^{n}\omega})]_{\mu_{i+1}}\rangle_{1\leqslant j\leqslant m_{i}-1}\big{)} =\mu_{i}.
\end{align*}
\end{proof}

\begin{definition}
 Let $X$ be a Banach space. We define $G(X)$ to be the Grassmanian of closed subspaces of $ X $ equipped with the Hausdorff distance
 \begin{align*}
    d_{H}(A,B):=\max\lbrace\sup_{a\in S_A}d(a,S_{B}),\sup_{b\in S_B}d(b,S_{A})\rbrace.
\end{align*}
    where $ S_{A}=\lbrace a\in A\ : \ \Vert a\Vert =1\rbrace $. Set 
    \begin{align*}
     G_{k}(X)=\lbrace A\in G(X)\ :\ \operatorname{dim}[A]=k\rbrace \quad \text{and} \quad G^{k}(X)=\lbrace A\in G(X)\ : \ \operatorname{dim}[X/A] =k \rbrace.
    \end{align*}
\end{definition}
It can be shown that $(G(X),d_H)$ is a complete metric space and that $G_k(X)$ and $G^k(X)$ are closed subsets \cite[Chapter IV]{Kat95}. The following lemma will be useful.

\begin{lemma}\label{Bli}
For $ A,B\in G(X) $ set
\begin{align*}
\delta(A,B):=\sup_{a\in S_{A}}d(a,B).
\end{align*} 
Then the following holds:
\begin{itemize}
\item[(i)] $ d_{H}(A,B)\leqslant 2\max\lbrace \delta(A,B),\delta(B,A)\rbrace $.
\item[(ii)]  If $ A , B\in G_{k}(X) $ with $ d(A,B)<\frac{1}{k} $ for some $k \in \N$, we have
\begin{align*}
\delta(B,A)\leqslant \frac{k\delta(A,B)}{1- k\delta(A,B)}.
\end{align*}
\end{itemize}
\end{lemma}
\begin{proof}
    \cite[Lemma 2.6]{Blu16}.
\end{proof}

\begin{proposition}\label{comp_}
Fix $ i\geqslant 1 $ and $\omega \in \tilde{\Omega}$. For every $ n\in \mathbb{Z} $, let $ H_{\sigma^{n}\omega}^{n}\subset F_{\mu_{i}}(\sigma^{n}\omega)$ be a complementary subspace for $ F_{\mu_{i+1}}(\omega) $ satisfying \eqref{CXCC}. Set $ \tilde{H}^{n}_{\omega}:=\psi^{n}_{\sigma^{n}\omega}(H^{n}_{\sigma^{n}\omega}) $. Then the sequence $ \lbrace \tilde{H}^{n}_{\omega}\rbrace_{n\geqslant 1} $ is Cauchy in $\big{(}G_{m_{i}}(F_{\mu_{i}}(\omega)),d_{H}\big{)} $ on a $\theta$-invariant set of full measure.
\end{proposition}

\begin{proof}
From \eqref{CXCC}, we can deduce that for every $  n\in\mathbb{N}$ and $ \xi_{\sigma^{n}\omega}\in S_{H^{n}_{\sigma^n\omega}}$,
\begin{align}\label{dis_norm}
  \frac{1}{M_{i}} < \Vert [\xi_{\sigma^{n}\omega}]_{\mu_{j+1}}\Vert \leq 1.
\end{align}
Note that $=\psi^{k}_{\sigma^{n}\omega} \vert_{H_{\sigma^{n}\omega}^{n}}$ is injective for any $k \geq 1$, therefore $\operatorname{dim}( \tilde{H}^{n}_{\omega}) = \operatorname{dim}( H_{\sigma^{n}\omega}^{n})= m_i$. Since $ \mu_{i+1} < \mu_{i} $, we know that $\tilde{H}^{n}_{\omega} \cap F_{\mu_{i+1}}(\omega) = \{0\}$ and since $ \operatorname{dim}[\frac{F_{\mu_{i}}(\omega)}{F_{\mu_{i+1}}(\omega)}] =m_{i} $, we obtain that
\begin{align*}
\tilde{H}^{n}_{\omega}\oplus F_{\mu_{i+1}}(\omega)=F_{\mu_{i}}(\omega)
 \end{align*}
 for any $n \in \N$. Let $\lbrace\xi^{j}_{\sigma^{n}\omega}\rbrace_{1\leqslant j\leqslant m_{i}} \subset S_{F_{\mu_{i}}(\sigma^{n}\omega)} $ be a base for $ H^{n}_{\sigma^{n}\omega}$. Then for $ \xi_{\sigma^{n+1}\omega}\in S_{F_{\mu_{i}}(\sigma^{n+1}\omega)} \cap H^{n+1}_{\sigma^{n+1}\omega}$, there exist $ \lbrace\beta_{j}\rbrace_{1\leqslant j\leqslant m_{i}}\subset\mathbb{R} $ such that
\begin{align*}
Z^{n}_{\omega} := \frac{\psi^{n+1}_{\sigma^{n+1}\omega}(\xi_{\sigma^{n+1}\omega})}{\Vert \psi^{n+1}_{\sigma^{n+1}\omega}(\xi_{\sigma^{n+1}\omega})\Vert}-\sum_{1\leqslant j\leqslant m_{i}}\beta_{j}\frac{\psi^{n}_{\sigma^{n}\omega}(\xi_{\sigma^{n}\omega}^{j})}{\Vert \psi^{n}_{\sigma^{n}\omega}(\xi_{\sigma^{n}\omega}^{j})\Vert}\in F_{\mu_{i+1}}(\omega).
\end{align*}
It follows that
\begin{align*}
 Y^{n}_{\sigma^{n}\omega} :=  \frac{\psi^{1}_{\sigma^{n+1}\omega}(\xi_{\sigma^{n+1}\omega})}{\Vert \psi^{n+1}_{\sigma^{n+1}\omega}(\xi_{\sigma^{n+1}\omega})\Vert}-\sum_{1\leqslant j\leqslant m_{i}}\beta_{j}\frac{\xi_{\sigma^{n}\omega}^{j}}{\Vert \psi^{n}_{\sigma^{n}\omega}(\xi_{\sigma^{n}\omega}^{j})\Vert}\in F_{\mu_{i+1}}(\sigma^{n}\omega),
\end{align*}
thus
\begin{align*}
\big{\Vert} \sum_{1\leqslant j\leqslant m_{i}}\beta_{j}\frac{\xi_{\sigma^{n}\omega}^{j}}{\Vert \psi^{n}_{\sigma^{n}\omega}(\xi_{\sigma^{n}\omega}^{j})\Vert}\big{\Vert}&\leqslant   \Vert\Pi_{H^{n}_{\sigma^{n}\omega}||F_{\mu_{j+1}}(\sigma^{n}\omega)}\Vert \frac{\Vert\psi^{1}_{\sigma^{n+1}\omega}\Vert}{\Vert\psi^{n+1}_{\sigma^{n+1}\omega}(\xi_{\sigma^{n+1}\omega})\Vert}\\&\leqslant M_{i} \frac{ \Vert\psi^{1}_{\sigma^{n+1}}\Vert}{\Vert\psi^{n+1}_{\sigma^{n+1}\omega}(\xi_{\sigma^{n+1}\omega})\Vert}
\end{align*}
and so
\begin{align}\label{first_}
 d\bigg{(}\frac{\psi^{n+1}_{\sigma^{n+1}\omega}(\xi_{\sigma^{n+1}\omega})}{\Vert \psi^{n+1}_{\sigma^{n+1}\omega}(\xi_{\sigma^{n+1}\omega})\Vert},\tilde{H}^{n}_{\omega}\bigg{)}\leqslant \| Z^{n}_{\omega} \| = \Vert\psi^{n}_{\sigma^{n}\omega}(Y^{n}_{\sigma^{n}\omega})\Vert  \leqslant \ (M_{i}+1)\frac{\Vert\psi^{n}_{\sigma^{n}\omega}|_{F_{\mu_{i+1}}(\sigma^n \omega)}\Vert  \Vert\psi^{1}_{\sigma^{n+1}\omega}\Vert}{\Vert\psi^{n+1}_{\sigma^{n+1}\omega}(\xi_{\sigma^{n+1}\omega})\Vert}.
\end{align}
Note that $ \lim_{n\rightarrow\infty}\frac{1}{n}\log\Vert\psi^{1}_{\sigma^{n}\omega}\Vert =0 $ from Birkhoff's Ergodic Theorem. Using Lemma \ref{lemma:reverse_sub} and \eqref{eqn:growth_slow_subspace} for $k = 1$, we have
\begin{align*}
 \limsup_{n \to \infty} \frac{1}{n} \log \Vert\psi^{n}_{\sigma^{n}\omega}|_{F_{\mu_{i+1}}(\sigma^n \omega)}\Vert \leq \mu_{i + 1}.
\end{align*}
From Lemma \ref{fra_fra} the estimate \ref{dis_norm} and Lemma \ref{Bli}, \eqref{first_} implies that for $ \epsilon >0 $ small and large $ n $,
\begin{align*}
d_{H}\big{(}\tilde{H}^{n}_\omega ,\tilde{H}^{n+1}_{\omega}\big{)}< M \exp\big{(}n(\mu_{i+1}-\mu_{i}+\epsilon)\big{)}
\end{align*}
for a constant $M > 0$. The claim is proved.
\end{proof}

Next, we collect some facts about the limit of the sequence above. 

\begin{lemma}\label{lemma:first_prop_H}
Assume $ \tilde{H}^{n}_{\omega}\xrightarrow{d_{H}}\tilde{H}_{\omega} $. Then the following holds:
\begin{itemize}
 \item[(i)] $ \tilde{H}_{\omega} $ is invariant, i.e. $\psi_{\omega}^k (\tilde{H}_{\omega}) = \tilde{H}_{\theta^k \omega}$ for any $k \geq 0$.
 \item[(ii)] $\tilde{H}_{\omega}\cap F_{\mu_{i+1}}(\omega) = \lbrace 0\rbrace$.
 \item[(iii)] $ \tilde{H}_{\omega} $ only depends on $ \omega $. In particular, it does not depend on the choice of the sequence $\{ \tilde{H}^{n}_{\omega} \}_{n \geq 1}$.
 \end{itemize}
\end{lemma}
\begin{proof}
By construction, $ \tilde{H}_{\omega} $ is invariant. We proceed with (ii). Consider the dual map
\begin{align*}
\big{(}\psi^{n}_{\sigma^{n}\omega}\big{)}^{*}_{\mu_{i}} :\big{(}F_{\mu_{i}}(\omega)\big{)}^{*}\rightarrow \big{(}F_{\mu_{i}}(\sigma^{n}\omega)\big{)}^{*}.
\end{align*}
It is straightforward to see that $ \big{(}\psi^{n}_{\sigma^{n}\omega}\big{)}^{*}_{\mu_{i}} $ enjoys the cocycle property. From \eqref{Du_Re} and \cite[Proposition 4.15]{GRS19}, we can find a closed subspace $ G^{*}_{\mu_{i+1}}(\omega)\subset \big{(}F_{\mu_{i}}(\omega)\big{)}^{*}  $ such that $ \operatorname{dim} [ (F_{\mu_{i}}(\omega))^{*} / G^{*}_{\mu_{i+1}}(\omega) ] = m_{i} $ and for $ \xi^{*}_{\omega}\in G^{*}_{\mu_{i+1}}(\omega) $, $ \limsup_{n\rightarrow\infty}\frac{1}{n}\log\big{\Vert}\big{(}\psi^{n}_{\sigma^{n}\omega}\big{)}^{*}_{\mu_{i}}( \xi^{*}_{\omega}) \big{\Vert} \leqslant\mu_{i+1}$. Set
\begin{align*}
\big{(}F_{\mu_{i+1}}(\omega)\big{)}_{\mu_{i}}^{\perp}=\big{\lbrace} \xi^{*}_{\omega}\in \big{(}F_{\mu_{i}}(\omega)\big{)}^{*}\ : \ \xi^{*}_{\omega}|_{F_{\mu_{i+1}}(\omega)} =0\big{\rbrace}.
\end{align*}
By Hahn-Banach separation theorem,
\begin{align*}
 \operatorname{dim} \left[ \big{(}F_{\mu_{i+1}}(\omega)\big{)}_{\mu_{i}}^{\perp} \right] = \operatorname{dim} \left[ F_{\mu_i}(\omega) / F_{\mu_{i+1}}(\omega) \right] = m_{i}.
\end{align*}
Let $ \xi^{*}_{\omega}\in \big{(}F_{\mu_{i+1}}(\omega)\big{)}_{\mu_{i}}^{\perp}\cap G^{*}_{\mu_{i+1}}(\omega)$ and assume that $ \xi^{*}_{\omega}\neq 0 $. Then for some $ \xi_{\omega}\notin F_{\mu_{i}}(\omega)\setminus F_{\mu_{i+1}}(\omega)  $, $\langle \xi^{*}_{\omega},\xi_{\omega} \rangle = 1 $. Using surjectivity of $[\psi^n_{\sigma^n \omega}]_{\mu_{i+1}}$, for every $ n\in\mathbb{N} $, we can find $ \xi_{\sigma^{n}\omega }\in H^{n}_{\sigma^{n}\omega} $ such that
\begin{align*}
\psi^{n}_{\sigma^{n}\omega}(\xi_{\sigma^{n}\omega})=\xi_{\omega}\ \ \ \text{mod } F_{\mu_{i+1}}(\omega).
\end{align*}
Consequently, $\langle (\psi^{n}_{\sigma^{n}\omega})_{\mu_{i}}^{*}(\xi^{*}_{\omega}),\xi_{\sigma^{n}\omega} \rangle =1 $. From Lemma \ref{fra_fra} ,
\begin{align}\label{GBT}
\lim_{n\rightarrow\infty}\frac{1}{n}\log\big{\Vert} \big{[}\psi^{n}_{\sigma^{n}\omega}(\frac{\xi_{\sigma^{n}\omega}}{\Vert [\xi_{\sigma^{n}\omega}]_{\mu_{i+1}}\Vert})\big{]}_{\mu_{i+1}}\big{\Vert} =\lim_{n\rightarrow\infty}\frac{1}{n}\log\big{\Vert}\frac{\Vert [\xi_{\omega}]_{\mu_{j+1}}\Vert}{\Vert [\xi_{\sigma^{n}\omega}]_{\mu_{j+1}}\Vert}\big{\Vert} =\mu_{i}.
\end{align}
%Consequently $ <(\psi^{n}_{\sigma^{n}\omega})_{\mu_{i}}^{*}(\xi^{*}_{\omega}),\xi_{\sigma^{n}\omega}> =1 $, from the \eqref{GBT}, 
Hence for $ \epsilon>0 $ and large $ n $,
\begin{align*}
\Vert [\xi_{\sigma^{n}\omega}]_{\mu_{j+1}}\Vert < \exp(-n\big{(}\mu_{i}-\epsilon)\big{)}
\end{align*}
which is a contradiction since $\Vert (\psi^{n}_{\sigma^{n}\omega})_{\mu_{i}}^{*}(\xi^{*}_{\omega})\Vert \leqslant \exp\big{(}n(\mu_{i+1}+\epsilon)\big{)}$. Thus we have shown that 
\begin{align}\label{decomp}
\big{(}F_{\mu_{i}}(\omega)\big{)}^{*}=\big{(}F_{\mu_{i+1}}(\omega)\big{)}_{\mu_{i}}^{\perp}\oplus G^{*}_{\mu_{i+1}}(\omega).
\end{align}
Now let $ \xi_{\omega}\in \tilde{H}_{\omega}\cap F_{\mu_{i+1}}(\omega)$ and assume that $ \Vert\xi_{\omega}\Vert =1 $. From \ref{decomp}, we can find $ \xi_{\omega}^{*}\in G^{*}_{\mu_{i+1}}(\omega) $ such that $ \langle \xi^{*}_{\omega} ,\xi_{\omega} \rangle =1 $. By definition of $\tilde{H}_{\omega}$, there exist $ \xi^{n}_{\sigma^{n}\omega}\in S_{H^{n}_{\sigma^{n}\omega}} $ such that $ \frac{\psi^{n}_{\sigma^{n}\omega}(\xi^{n}_{\sigma^{n}\omega})}{\Vert \psi^{n}_{\sigma^{n}\omega}(\xi^{n}_{\sigma^{n}\omega})\Vert}\rightarrow \xi_{\omega} $ as $n \to \infty$, and consequently
\begin{align*}
\langle \xi^{*}_{\omega},\frac{\psi^{n}_{\sigma^{n}\omega}(\xi^{n}_{\sigma^{n}\omega})}{\Vert \psi^{n}_{\sigma^{n}\omega}(\xi^{n}_{\sigma^{n}\omega})\Vert} \rangle = \langle(\psi^{n}_{\sigma^{n}\omega})^{*}(\xi^{*}_{\omega}),\frac{\xi^{n}_{\sigma^{n}\omega}}{\Vert\psi^{n}_{\sigma^{n}\omega}(\xi^{n}_{\sigma^{n}\omega})\Vert} \rangle \rightarrow 1
\end{align*}
as $n \to \infty$. With Lemma \ref{fra_fra} and a similar argument as above, this is again a contradiction and we have shown (ii).
% Which again is again contradiction by lemma \ref{fra_fra}, thus $ F_{\mu_{i}}(\omega)=\tilde{H}_{\omega}\oplus F_{\mu_{i+1}}(\omega) $, 
It remains to prove (iii). For $ \xi_{\omega}\in \tilde{H}_{\omega}\subset (F_{\mu_{i}}(\omega))^{**} $, $ \xi^{*}_{\omega}\in G_{\mu_{i+1}}^{*}(\omega) $ and a sequence $ \xi^{n}_{\sigma^{n}\omega}$ chosen as above, 
\begin{align*}
\langle \frac{\psi^{n}_{\sigma^{n}\omega}(\xi^{n}_{\sigma^{n}\omega})}{\Vert \psi^{n}_{\sigma^{n}\omega}(\xi^{n}_{\sigma^{n}\omega})\Vert},\xi^{*}_{\omega} \rangle \rightarrow 0
\end{align*}
as $n \to \infty$. Therefore, $\tilde{H}_{\omega}\subset \big{(}G_{\mu_{i+1}}^{*}(\omega)\big{)}_{\mu_{i}}^{\perp}=\big{\lbrace} \xi^{**}_{\omega}\in \big{(}F_{\mu_{i}}(\omega)\big{)}^{**}\ : \ \xi^{**}_{\omega}|_{G_{\mu_{i+1}}^{*}(\omega)} =0\big{\rbrace} $ and since $\operatorname{dim} \big{[} \big{(}G_{\mu_{i+1}}^{*}(\omega)\big{)}_{\mu_{i}}^{\perp}\big{]} =m_{i} $, we obtain 
\begin{align}\label{eqn:uniqueness_H}
 \tilde{H}_{\omega}=\big{(}G_{\mu_{i+1}}^{*}(\omega)\big{)}_{\mu_{i}}^{\perp}
\end{align}
which proves (iii).
\end{proof}

So far, we have shown the following: There is a $\theta$-invariant set $\tilde{\Omega} \subset \Omega$ of full measure such that for every $i \geq 1$ with $\mu_i > \mu_{i+1}$ and $\omega \in \tilde{\Omega}$, there is an $m_i$-dimensional subspace $H^i_\omega$ such that
\begin{itemize}
 \item $H_{\omega}^i \oplus F_{\mu_{i+1}}(\omega) = F_{\mu_i}(\omega)$ and
 \item $\psi^n_{\omega}(H^i_{\omega}) = H^i_{\theta^n \omega}$.
\end{itemize}
 In particular, $\psi^n_{\omega} \vert_{H^i_{\omega}}$ is injective for every $n \geq 0$.

In the remaining part of this section, we study further properties of the spaces $H^i_{\omega}$. We start with a measurability result.

\begin{lemma}\label{Proj_}
For every $ i \geq 1 $ the maps
\begin{align*}
f_{1}(\omega) := \Vert\Pi_{{H}^{i}_{\omega}||F_{\mu_{i+1}}(\omega)}\Vert \ \ \text{and} \ \ f_{2}(\omega):=\Vert\Pi_{F_{\mu_{i+1}}(\omega)||{H}^{i}_{\omega}}\Vert
\end{align*}
are measurable.
\end{lemma}
\begin{proof}
We prove the claim for $i = 1$ first, i.e. $ \operatorname{dim}[E_{\omega} / F_{\mu_{2}}(\omega)] = m_1 $ for $\omega\in \tilde{\Omega}$. Let
\begin{align*}
 \{ (g_{k_1}, \ldots, g_{k_{m_1}})\, :\, k \in \N\} = \Delta_0^{m_1}.
\end{align*}
Fix $n \in \N$ and $\omega\in \tilde{\Omega}$. We define $ \lbrace U^{k}_{\sigma^{n}\omega}\rbrace_{k\geqslant 1} $ to be the family of subspaces of $ E_{\sigma^{n}\omega} $ given by $ U^{k}_{\sigma^{n}\omega} = \langle g_{k_{i}}(\sigma^{n}\omega) \rangle _{1\leqslant i\leqslant m_{1}, g_{k_{i}}\in\Delta_{0}} $, By the same technique as before (cf. e.g. the proof to Lemma \ref{mea_Vol}), the map
\begin{equation*}
\omega \mapsto G_{k}(\sigma^{n}\omega)=
\begin{cases}
\Vert \Pi_{U^{k}_{\sigma^{n}\omega}||F_{\mu_{2}}(\sigma^{n}\omega)}\Vert &\quad\ U^{k}_{\sigma^{n}\omega}\oplus F_{\mu_{2}}(\sigma^{n}\omega)=F_{\mu_{1}}(\sigma^{n}\omega)\\
\infty &\quad\ \text{otherwise}
\end{cases}
\end{equation*}
 is measurable. Set $ \psi_{n}(\omega):=\inf\lbrace k: G_{k}(\sigma^{n}\omega)<M_{1}\rbrace$ which is clearly measurable. By Proposition \ref{comp_}, $ \tilde{H}^{n}_{\omega} := \psi^{n}_{\sigma^{n}\omega}\big{(}U^{\psi_{n}(\omega)}_{\sigma^{n}\omega}\big{)} \xrightarrow{d_{H}}{H}_{\omega}^{j}$ and consequently $ \Pi_{\tilde{H}^{n}_{\omega}||F_{\mu_{2}}(\omega)}\rightarrow \Pi_{H^{j}_{\omega}||F_{\mu_{2}}(\omega)} $ as $n \to \infty$. Thus it is enough to show that for every $ g\in\Delta$, 
\begin{align}\label{eqn:proj_meas_interm}
    \omega \mapsto \Vert \Pi_{\tilde{H}^{n}_{\omega}||F_{\mu_{2}}(\omega)}g({\omega})\Vert
\end{align}
is measurable. Let $ \tilde{H}^{n}_{\omega}=\langle\psi^{n}_{\sigma^{n}\omega}(g_{i}(\sigma^n \omega))\rangle_{ 1\leqslant i\leqslant m_{1} }$, therefore,
\begin{align*}
    \Pi_{\tilde{H}^{n}_{\omega}||F_{\mu_2}(\omega)}g({\omega}) = \sum_{1\leqslant t\leqslant m_{1}}\alpha_{t}(\omega) \psi^{n}_{\sigma^{n}\omega}(g_{t}(\sigma^n \omega)).
\end{align*}
 We have to prove that each $\omega \mapsto \alpha_{t}(\omega) \in \R$ is measurable. Assume $ m_{1}=1 $ first. Since $g(\omega) - \alpha_1(\omega) \psi^n_{\sigma^n \omega}(g_1(\sigma^n \omega)) \in F_{\mu_2}(\omega)$, we have $\| [g(\omega)]_{\mu_2} \| = |\alpha_1(\omega)| \| [\psi^n_{\sigma^n \omega}(g_1(\sigma^n \omega))]\|$ and therefore
 \begin{align*}
  |\alpha_1(\omega)| = \frac{d\big{(}g(\omega),F_{\mu_{2}}(\omega)\big{)}}{d\big{(}\psi_{\sigma^{n}\omega}(g_{1}(\sigma^{n}\omega)),F_{\mu_{2}}(\omega)\big{)}}
 \end{align*}
    Set
\begin{align*}
d_{0}(\omega) := d\big{(}g(\omega),F_{\mu_{2}}(\omega)\big{)}  \quad \text{and} \quad  d_{1}(\omega) := d\big{(}\psi_{\sigma^{n}\omega}(g_{1}(\sigma^{n}\omega)),F_{\mu_{2}}(\omega)\big{)}.
\end{align*}
As before (cf. Lemma \ref{mea_distance}), we can see that $ d_{0}(\omega)$ and $ d_{1}(\omega) $ are measurable, and we have 
\begin{align*}
 \Pi_{\tilde{H}^{n}_{\omega}||F_{\mu_2}(\omega)}g({\omega}) = G(\omega)  \frac{d_0(\omega)}{d_1(\omega)} \psi^{n}_{\sigma^{n}\omega}(g_{1}(\sigma^n \omega))
\end{align*}
where $G(\omega)$ takes values in $\{-1,0,1\}$. 
% also $ g({\omega})-(-1)^{G(\omega)}\frac{d_{0}(\omega)}{d_{1}(\omega)}\psi^{n}_{\sigma^{n}\omega}(g_{1}(\omega))\in F_{\mu_{2}}(\omega) $, where $ G(\omega) $ takes value in $ \lbrace 1,2\rbrace $, 
Set $ h_{0}(\omega) := g(\omega)-\frac{d_{0}(\omega)}{d_{1}(\omega)}\psi^{n}_{\sigma^{n}\omega}(g_{1}(\sigma^{n}\omega)) $ and $h_{1}(\omega) := g(\omega)+\frac{d_{0}(\omega)}{d_{1}(\omega)}\psi^{n}_{\sigma^{n}\omega}(g_{1}(\sigma^{n}\omega))$ and define
\begin{align*}
J_{0}(\omega):=\lim_{m\rightarrow \infty}\frac{1}{m}\log\big{\Vert}\psi^{m}_{\omega}\big{(}h_{0}(\omega)\big{)}\big{\Vert} , \ \ \ \ \ J_{1}(\omega):=\lim_{m\rightarrow \infty}\frac{1}{m}\log\big{\Vert}\psi^{m}_{\omega}\big{(}h_{1}(\omega)\big{)}\big{\Vert}.
\end{align*}
It follows that $ J_{0} $ and $ J_{1}$ are measurable. Finally, 
\begin{align*}
 \Pi_{\tilde{H}^{n}_{\omega}||F_{\mu}(\omega)}g({\omega}) = (1 - \chi_{\{ g(\omega) \in F_{\mu_2}(\omega)\}}) \left[ g(\omega) - \chi_{\mu_{2}} \big{(}J_0(\omega)\big{)}h_0(\omega) - \chi_{\mu_{2}}\big{(}J_{1}(\omega)\big{)} h_{1}(\omega) \right]
\end{align*} 
which proves measurability of \eqref{eqn:proj_meas_interm} for $m_1 = 1$. For $ m_{1}>1 $, we invoke the same technique: Let
\begin{align*}
&d_{0}(\omega)=d\big{(}g(\omega),F_{\mu_{2}}(\omega)\oplus\langle\psi^{n}_{\sigma^{n}\omega}(g_{t}(\sigma^{n}\omega))\rangle_{{2\leqslant t\leqslant m_{1}}}\big{)} ,  \\ &d_{1}(\omega)=d\big{(}\psi^{n}_{\sigma^{n}\omega}(g_{1}(\sigma^{n}\omega)),F_{\mu_{2}}(\omega)\oplus\langle\psi^{n}_{\sigma^{n}\omega}(g_{j}(\sigma^{n}\omega))\rangle_{{2\leqslant t\leqslant m_{1}}}\big{)}.
\end{align*}
For $ h_{0}(\omega)=g(\omega)-\frac{d_{0}(\omega)}{d_{1}(\omega)}\psi^{n}_{\sigma^{n}\omega}(g_{1}(\sigma^{n}\omega))  $ and $ h_{1}(\omega)=g(\omega)+\frac{d_{0}(\omega)}{d_{1}(\omega)}\psi^{n}_{\sigma^{n}\omega}(g_{1}(\sigma^{n}\omega)) $ let 
\begin{align*}
&d_{i0}(\omega):=d\big{(}h_{i}(\omega),F_{\mu_{2}}(\omega)\oplus\langle\psi^{n}_{\sigma^{n}\omega}(g_{t}(\sigma^{n}\omega))\rangle_{{3\leqslant t\leqslant m_{1}}}\big{)}, \ i\in\lbrace 0,1\rbrace\\&d_{01}(\omega)=d_{11}(\omega)=d\big{(}\psi^{n}_{\sigma^{n}\omega}(g_{2}(\sigma^{n}\omega)),F_{\mu_{2}}(\omega)\oplus\langle\psi^{n}_{\sigma^{n}\omega}(g_{t}(\sigma^{n}\omega))\rangle_{{3\leqslant t\leqslant m_{1}}}\big{)}.
\end{align*}
For $ i\in\lbrace 0,1\rbrace $ define 
\begin{align*}
h_{0,i}=h_{0}(\omega)+(-1)^{i+1}\frac{d_{00}(\omega)}{d_{01}(\omega)}\psi^{n}_{\sigma^{n}\omega}(g_{2}(\sigma^{n}\omega))\\
h_{1,i}=h_{1}(\omega)+(-1)^{i+1}\frac{d_{10}(\omega)}{d_{11}(\omega)}\psi^{n}_{\sigma^{n}\omega}(g_{2}(\sigma^{n}\omega)).
\end{align*}
We repeat the same procedure with our four new functions. Iterating this, we end up with $ 2^{m_1} $ functions $ \lbrace I_{t}(\omega)\rbrace_{1\leqslant t\leqslant 2^{m_{1}}} $ for which we define $ J_{t}(\omega) := \lim_{m\rightarrow\infty}\frac{1}{m}\log\big{\Vert}\psi^{m}_{\omega}(I_{t}(\omega))\big{\Vert}$. Since
\begin{align*}
 \Pi_{\tilde{H}^{n}_{\omega}||F_{\mu}(\omega)}g({\omega}) =  (1 - \chi_{\{ g(\omega) \in F_{\mu_2}(\omega)\}}) \left[ g(\omega)-\sum_{0\leqslant t\leqslant 2^{m_{1}}}\chi_{\mu_{2}}\big{(}J_{t}(\omega)\big{)}I_{t}(\omega) \right],
\end{align*} 
measurability of \eqref{eqn:proj_meas_interm} follows for arbitrary $m_1$. As a consequence,
\begin{align}\label{eqn:dense_subset_proj}
    \big{\lbrace}\Pi_{F_{\mu_{2}}(\omega)||H^{1}_{\omega}}\big{(}g(\omega)\big{)}\,:\, g\in\Delta_{0} \big{\rbrace}
\end{align}
 is a dense subset of $ F_{\mu_{2}}(\omega) $ and for $ g\in \Delta $ and $ k \geq 0 $,
\begin{align*}
\omega \mapsto \bigg{\Vert}{\psi}^{k}_{\omega}\bigg{(}\Pi_{F_{\mu_{2}}(\omega)||H^{1}_{\omega}}\big{(}g(\omega)\big{)}\bigg{)}\bigg{\Vert} 
\end{align*}
is measurable. For $k = 0$, we obtain measurability of $f_2$ for $i = 1$. We can now repeat the argument above for $i = 2$ using the dense subset in $\eqref{eqn:dense_subset_proj}$ instead of $\Delta_0$ to see that $f_1$ and $f_2$ are also measurable for $i = 2$. The general case follows by induction.
\end{proof}

\begin{remark}
With the same strategy as in Lemma \ref{Proj_}, we can see that for each $ 1 \leq l \leq j $ and $k \geq 0$,
\begin{align*}
f_{1}(\omega) := \big{\Vert}\Pi_{\oplus_{l \leqslant i < j}H_{\omega}^{i}\oplus F_{\mu_{j}}(\omega)}\big{\Vert}, \ f_{2}(\omega) := \big{\Vert}\Pi_{F_{\mu_{j}}(\omega)||\oplus_{l \leqslant i < j}H^{i}_{\omega}}\big{\Vert} \ \text{and} \ f_{3}(\omega) := \Vert\psi^{k}_{\omega}|_{\oplus_{l \leqslant i < j}H^{i}_{\omega}}\Vert
\end{align*}
are measurable.
\end{remark}

\begin{lemma}\label{inve_lim}
For a measurable and non-negative function $f :\Omega\rightarrow\mathbb{R} $ 
\begin{align*}
\lim_{n\rightarrow\infty}\frac{1}{n}f(\theta^{n}\omega)=0 \text{ a.s.}\ \ \ \text{if and only if} \ \ \ \lim_{n\rightarrow\infty}\frac{1}{n}f(\sigma^{n}\omega)=0 \text{ a.s.}
\end{align*}
\end{lemma}

\begin{proof}
The main idea is due to Jack Feldman, cf. \cite[Lemma 7.2]{LPP95}. Assume that $ \lim_{n\rightarrow\infty}\frac{1}{n}f(\theta^{n}\omega)=0 $ on a set of full measure $\Omega^0$. Let $\epsilon > 0$ and set
\begin{align*}
\Omega_{n}:=\lbrace \omega \in \Omega^0 \, :\,  \forall i\geqslant n \ \ \frac{f(\theta^{i}\omega)}{i}\leqslant \epsilon\rbrace.
\end{align*}
Fom our assumptions, for some $ n_{0} \in \N  $,
\begin{align*}
    \P(\Omega_{n_{0}})>\frac{9}{10}.
\end{align*}
From Birkhoff's ergodic theorem, there is a set of full measure $\Omega^1$ such that for every $ \omega\in \Omega^1$, we can find $ m_{0}=m_{\omega} $ such that for $ m\geqslant m_{0} $,
\begin{align}\label{ZZS}
\frac{1}{m}\sum_{0\leqslant j\leqslant m}\chi_{\Omega_{n_{0}}}(\sigma^{j}\omega) > \frac{9}{10}.
\end{align}
W.l.o.g., we may assume that $\Omega^0 = \Omega^1$. Now for $ k\geqslant \max \lbrace 3n_{0}, m_{0}\rbrace $, set $ m=\lfloor\frac{5}{3}k\rfloor +1$. Then from \eqref{ZZS}
\begin{align*}
\frac{1}{m}\big{[}\sum_{0\leqslant j\leqslant \frac{4m}{5}}\chi_{\Omega_{n_{0}}}(\sigma^{j}\omega)+\sum_{\frac{4m}{5}<j\leqslant m}\chi_{\Omega_{n_{0}}}(\sigma^{j}\omega)\big{]}>\frac{9}{10}.
\end{align*}
Consequently, there exists $ \frac{4m}{5}<j\leqslant m $ such that $ \sigma^{j}\omega\in\Omega_{n_{0}} $. Set $ i := j-k >n_{0}$. Then by the definition of $ \Omega_{n_{0}} $,
\begin{align*}
\frac{f(\theta^{i}\sigma^{j}\omega)}{i}=\frac{f(\sigma^{k}\omega)}{j-k}\leqslant \epsilon.
\end{align*}
Since $ j-k \leq \frac{2}{3}k + 1 $ and $ \epsilon $ is arbitrary, our claim is shown. The other direction can be proved similarly. 
\end{proof}

As a consequence, we obtain the following:
\begin{lemma}\label{lemma:pro_invv}
 For each $ 1 \leq l \leq j $ and $\omega \in \tilde{\Omega}$,
 \begin{align}\label{pro_invv}
\lim_{n\rightarrow\infty}\frac{1}{n}\log\Vert\Pi_{\oplus_{l\leqslant i<j}H^{i}_{\theta^{n}\omega}||F_{\mu_{j}}(\theta^{n}\omega)}\Vert=\lim_{n\rightarrow\infty}\frac{1}{n}\log\Vert\Pi_{\oplus_{l\leqslant i<j}H^{i}_{\sigma^{n}\omega}||F_{\mu_{j}}(\sigma^{n}\omega)}\Vert =0.
\end{align}
\end{lemma}

\begin{proof}
 Follows from a straightforward generalization of \cite[Lemma 4.4]{GRS19} and Lemma \ref{inve_lim}.
\end{proof}

The following lemma characterizes the spaces $H^{i}_{\omega}$ as `fast' growing subspaces.

\begin{proposition}\label{In_In_pr}
For $\omega \in \tilde{\Omega}$, every $ i\geqslant N $ and $ \xi_{\omega}\in H^{i}_{\omega} $,
\begin{align}\label{eqn:forward_fast_growing}
 \lim_{n\rightarrow\infty}\frac{1}{n}\log\Vert \psi^{n}_{\omega}(\xi_{\omega}) \Vert  = \lim_{n\rightarrow\infty}\frac{1}{n}\log\Vert \psi^{n}_{\omega} |_{H^{i}_{\omega}}\Vert = \mu_{i}
\end{align}
and 
\begin{align}\label{eqn:backward_fast_growing}
\lim_{n\rightarrow\infty}\frac{1}{n}\log\Vert(\psi^{n}_{\sigma^{n}\omega})^{-1}(\xi_{\omega})\Vert = \lim_{n\rightarrow\infty}\frac{1}{n}\log\Vert(\psi^{n}_{\sigma^{n}\omega}|_{H^{i}_{\omega}} )^{-1} \Vert = -\mu_{i}.
\end{align}

\end{proposition}

\begin{proof}
The equalities \eqref{eqn:forward_fast_growing} follow by applying the  Multiplicative Ergodic Theorem \cite[Theorem 4.17]{GRS19} to the map $\psi^n_\omega \vert_{H^i_{\omega}} \colon H^i_{\omega} \to H^i_{\theta^n \omega}$. It remains to prove \eqref{eqn:backward_fast_growing}. By definition, for every $ \xi_{\omega}\in H^{i}_{\omega} $,
\begin{align*}
\frac{\Vert(\psi^{n}_{\sigma^{n}\omega})^{-1}(\xi_{\omega})\Vert}{\Vert[\xi_{\omega}]_{\mu_{i+1}}\Vert} \times \frac{\big{\Vert}\big{[}\psi^{n}_{\sigma^{n}\omega}\big{(}(\psi^{n}_{\sigma^{n}\omega})^{-1}(\xi_{\omega}) \big{)} \big{]}_{\mu_{i+1}}\big{\Vert}}{\Vert [(\psi^{n}_{\sigma^{n}\omega})^{-1}(\xi_{\omega})]_{\mu_{i+1}}\Vert} = \frac{\Vert(\psi^{n}_{\sigma^{n}\omega})^{-1}(\xi_{\omega})\Vert}{\Vert [(\psi^{n}_{\sigma^{n}\omega})^{-1}(\xi_{\omega})]_{\mu_{i+1}}\Vert} \leqslant \Vert \Pi_{H^{i}_{\sigma^{n}\omega}||F_{\mu_{i+1}}(\sigma^{n}\omega)}\Vert.
\end{align*}
From Lemma \ref{fra_fra},
\begin{align*}
\lim_{n\rightarrow\infty}\frac{1}{n}\inf _{\bar{\xi}_{\sigma^{n}\omega}\in H^{i}_{\sigma^{n}\omega}}\frac{\Vert[\psi^{n}_{\sigma^{n}\omega}(\bar{\xi}_{\sigma^{n}\omega})]_{\mu_{i+1}}\Vert}{\Vert [\bar{\xi}_{\sigma^{n}\omega}]_{\mu_{i+1}}\Vert} = \lim_{n\rightarrow\infty}\frac{1}{n} \frac{\Vert[\psi^{n}_{\sigma^{n}\omega}(\hat{\xi}_{\sigma^{n}\omega})]_{\mu_{i+1}}\Vert}{\Vert [\hat{\xi}_{\sigma^{n}\omega}]_{\mu_{i+1}}\Vert} = \mu_{i}
\end{align*}
where $\hat{\xi}_{\sigma^{n}\omega} \in H^i_{\sigma^n \omega}$ is chosen such that
\begin{align*}
   \frac{\Vert[\psi^{n}_{\sigma^{n}\omega}(\hat{\xi}_{\sigma^{n}\omega})]_{\mu_{i+1}}\Vert}{\Vert [\hat{\xi}_{\sigma^{n}\omega}]_{\mu_{i+1}}\Vert} = \min _{\bar{\xi}_{\sigma^{n}\omega}\in H^{i}_{\sigma^{n}\omega}} \frac{\Vert[\psi^{n}_{\sigma^{n}\omega}(\bar{\xi}_{\sigma^{n}\omega})]_{\mu_{i+1}}\Vert}{\Vert [\bar{\xi}_{\sigma^{n}\omega}]_{\mu_{i+1}}\Vert}.
\end{align*}
Consequently, from \eqref{pro_invv},
\begin{align*}
\limsup_{n\rightarrow\infty}\frac{1}{n}\log\Vert(\psi^{n}_{\sigma^{n}\omega} |_{H^{i}_{\omega}} )^{-1}\Vert\leqslant -\mu_{i}
\end{align*}
Finally, from inequality $ \Vert\xi_{\omega}\Vert\leqslant \Vert\psi^{n}_{\sigma^{n}\omega}|_{H^{i}_{\sigma^{n}\omega}}\Vert \Vert (\psi^{n}_{\sigma^{n}\omega})^{-1}(\xi_{\omega})\Vert $, Lemma \ref{lemma:reverse_sub} and \eqref{eqn:forward_fast_growing}, the equalities \eqref{eqn:backward_fast_growing} can be deduced.
\end{proof}

\begin{lemma}\label{lemma:dist_forw_backw}
 Let $\omega \in \tilde{\Omega}$ and $i < k$. For every $i \leq j < k$, let $\{\xi^t_{\omega} \}_{t \in I_j}$ be a basis of $H^j_{\omega}$. Set $I := \cup_{i \leq j < k} I_j$ and assume $\xi^t_{\omega} \in H^j_{\omega}$. Then
 \begin{align}\label{eqn:dist_forw}
  \lim_{n \to \infty} \frac{1}{n} \log d (\psi^n_{\omega}(\xi^t_{\omega}), \langle \psi^n_{\omega}(\xi^{t'}_{\omega}) \rangle_{t' \in I\setminus \{t\}} ) = \mu_j
 \end{align}
 and
 \begin{align}\label{eqn:dist_backw}
  \lim_{n \to \infty} \frac{1}{n} \log d ((\psi^n_{\sigma^n \omega})^{-1} (\xi^t_{\omega}), \langle (\psi^n_{\sigma^n \omega})^{-1} (\xi^{t'}_{\omega}) \rangle_{t' \in I\setminus \{t\}} ) = - \mu_j.
 \end{align}

\end{lemma}

\begin{proof}
 We will prove \eqref{eqn:dist_backw} only, the proof for \eqref{eqn:dist_forw} is completely analogous. First, we claim that the statement is true for $j = i$ and $k = i+1$. Indeed, in this case we have the inequalities
 \begin{align*}
\frac{1}{\Vert\psi^{n}_{\sigma^{n}\omega}|_{H^{i}_{\sigma^{n}\omega}}\Vert}\leqslant \frac{d\big{(}(\psi^{n}_{\sigma^{n}\omega})^{-1}(\xi^{t}_{\omega}),\langle (\psi^{n}_{\sigma^{n}\omega})^{-1}(\xi^{t^{\prime}}_{\omega})\rangle_{t^{\prime}\in I\setminus\lbrace t\rbrace }\big{)}}{d\big{(}\xi^{t}_{\omega},\langle\xi^{t^{\prime}}_{\omega}\rangle_{t^{\prime}\in I\setminus\lbrace t\rbrace}\big{)}}\leqslant \Vert (\psi^{n}_{\sigma^{n}\omega})^{-1}|_{H^{i}_{\omega}}\Vert
\end{align*}
and we can conclude with Proposition \ref{In_In_pr}. For arbitrary $k$ and $j = i$, we can use the inequalities
\begin{align*}
1\leqslant\frac{d\big{(}(\psi^{n}_{\sigma^{n}\omega})^{-1}(\xi^{t}_{\omega}),\langle (\psi^{n}_{\sigma^{n}\omega})^{-1}(\xi^{t^{\prime}}_{\omega})\rangle_{t^{\prime}\in I_{i}\setminus\lbrace t\rbrace }\big{)}}{d\big{(}(\psi^{n}_{\sigma^{n}\omega})^{-1}(\xi^{t}_{\omega}),\langle (\psi^{n}_{\sigma^{n}\omega})^{-1}(\xi^{t^{\prime}}_{\omega})\rangle_{t^{\prime}\in I\setminus\lbrace t\rbrace }\big{)}}\leqslant \Vert\Pi_{H^{i}_{\sigma^{n}\omega}||F_{\mu_{i+1}}(\sigma^{n}\omega)}\Vert,
\end{align*}
Lemma \ref{lemma:pro_invv} and our previous result above. The definition of $\operatorname{Vol}$ allows to deduce that
\begin{align}\label{BBVC}
\lim_{n\rightarrow\infty} \frac{1}{n} \log \operatorname{Vol}\big{(}\big{(}(\psi^{n}_{\sigma^n \omega})^{-1}(\xi^{t}_{\omega})\big{)}_{t\in I_{k-1}},...,\big{(}(\psi^{n}_{\sigma^n \omega})^{-1}(\xi^{t}_{\omega})\big{)}_{t\in I_{i}}\big{)}=\sum_{i\leqslant j<k}-\mu_{j}\vert I_{j}\vert.
\end{align}
Since $\operatorname{Vol}$ is symmetric up to a constant, the claim \eqref{eqn:dist_backw} follows for arbitrary $j$.

\end{proof}

The following theorem summarizes the main result of this section.

\begin{theorem}\label{BT}
 There is a $\theta$-invariant set of full measure $\tilde{\Omega}$ such that for every $i \geq 1$ with $\mu_i > \mu_{i+1}$ and $\omega \in \tilde{\Omega}$, there is an $m_i$-dimensional subspace $H^i_\omega$ with the following properties:
\begin{itemize}
 \item[(i)] (Invariance)\ \ $\psi_{\omega}^k(H^i_{\omega}) = H^i_{\theta^k \omega}$ for every $k \geq 0$.
 \item[(ii)] (Splitting)\ \ $H_{\omega}^i \oplus F_{\mu_{i+1}}(\omega) = F_{\mu_i}(\omega)$. In particular,
 \begin{align*}
  E_{\omega} = H^1_{\omega} \oplus \cdots \oplus H^i_{\omega} \oplus   F_{\mu_{i+1}}(\omega).
 \end{align*}
 \item[(iii)] ('Fast' growing subspace I)\ \ For each $ h_{\omega}\in H^{i}_{\omega} $,
 \begin{align*}
  \lim_{n\rightarrow\infty}\frac{1}{n}\log\Vert \psi^{n}_{\omega}(h_{\omega})\Vert = \mu_{j}.
 \end{align*}
 \item[(iv)] ('Fast' growing subspace II)\ \ For each $ h_{\omega}\in H^{i}_{\omega} $,
\begin{align*}
\lim_{n\rightarrow\infty}\frac{1}{n}\log\Vert (\psi^{n}_{\sigma ^{n}\omega})^{-1}(h_{\omega})\Vert =-\mu_{j}.
\end{align*}
  \item[(v)] If $\{\xi^{t}_{\omega} \}_{1\leqslant t\leqslant m}$ is a basis of $ \oplus_{1\leqslant i\leqslant j}H^{i}_{\omega}$, then
\begin{align}
\begin{split}
&\lim_{n\rightarrow\infty}\frac{1}{n}\log \operatorname{Vol}\big{(}\psi^{n}_{\omega}(\xi^{1}_{\omega}),...,\psi^{n}_{\omega}(\xi^{m}_{\omega})\big{)} = \sum_{1\leqslant i\leqslant j}m_{i}\mu_{i} \quad \text{and} \\
&\lim_{n\rightarrow\infty}\frac{1}{n}\log \operatorname{Vol}\big{(}(\psi^{n}_{\sigma^{n}\omega})^{-1}(\xi^{1}_{\omega}),...,(\psi^{n}_{\sigma^{n}\omega})^{-1}(\xi^{m}_{\omega})\big{)}=\sum_{1\leqslant i\leqslant j}-m_{i}\mu_{i}.
\end{split}
\end{align}

\end{itemize}

\end{theorem}

\begin{proof}
 Properties (i) and (ii) are proven in Lemma \ref{lemma:first_prop_H}. (iii) and (iv) are shown in Proposition \ref{In_In_pr} and (v) can be deduced from Lemma \ref{lemma:dist_forw_backw}, using the definition of $\operatorname{Vol}$ and symmetry modulo a constant of this function.
\end{proof}

\begin{remark}
 Property (iv) seems to be new in the context of Banach spaces. Note that properties (i) - (iv) uniquely determine the spaces  $H^i_\omega$. In fact, an inspection of the proof of Lemma \ref{lemma:first_prop_H} reveals that these properties are sufficient to establish the equality \eqref{eqn:uniqueness_H}.
\end{remark}

 \section{Invariant Manifolds}\label{sec_inf_mfd}
 Let $ \lbrace E_{\omega}\rbrace_{\omega\in \Omega} $ be a measurable field of Banach spaces and $ \varphi^{n}_{\omega} $ a nonlinear cocycle on acting on it, i.e.
\begin{align*}
&\varphi^{n}_{\omega} \colon E_{\omega} \to E_{\theta^{n}\omega}\\
&\varphi^{n+m}_{\omega}(.) = \varphi^{n}_{\theta^{m}\omega}\big{(}\varphi^{m}_{\omega}(.)\big{)}.
\end{align*}

\begin{definition}
 
We say that $ \varphi^{n}_{\omega} $ admits a \emph{stationary solution} if there exists a map $Y:\Omega\longrightarrow \prod_{\omega\in\Omega}E_{\omega}$ such that
\begin{itemize}
\item[(i)]$Y_{\omega}\in E_{\omega}$,
\item[(ii)]$\varphi^n_\omega(Y_{\omega})=Y_{\theta^{n}\omega}$ and
\item[(iii)]$\omega\rightarrow\Vert Y_{\omega}\Vert $ is measurable.
\end{itemize}
\end{definition}

Stationary solutions should be thought of random analogues to fixed points in (deterministic) dynamical systems. If $ \varphi^{n}_{\omega}$ is Fr\'echet differentiable, one can easily check that the derivative around a stationary solution also enjoys the cocycle property, i.e for $ \psi^n_\omega(.) = D_{Y_{\omega}}\varphi^{n}_{\omega}(.) $, one has
\begin{align*}
\psi^{n+m}_{\omega}(.)=\psi^{n}_{\theta^{m}\omega}\big{(}\psi^m_\omega(.)\big{)}.
\end{align*} 
In the following, we will assume that $\varphi$ is Fr\'echet differentiable, that there exists a stationary solution $Y$ and that the linearized cocycle $\psi$ around $Y$ is compact and satisfies Assumption \ref{MES}. Furthermore, we will assume that 
\begin{align*}
\log^{+}\Vert \psi_{\omega} \Vert\in L^{1}(\Omega).
\end{align*}
Therefore, we can apply the MET to $\psi$. In the following, we will use the same notation as in the previous section.

\subsection{Stable manifolds}

\begin{definition}\label{stable-dfn}
Let $ Y $ be a  stationary solution, let $\lbrace ...<\mu_{j} < \mu_{j-1}<...<\mu_{1} \rbrace \in [-\infty,\infty)$ be the corresponding Lyapunov spectrum and $\tilde{\Omega}$ the $\theta$-invariant set on which the MET holds. Set $ \mu_{j_{0}} = \max \lbrace \mu_{j} : \mu_{j} < 0 \rbrace $ and $ \mu_{j_{0}} = -\infty $ if all finite $ \mu_{j} $ are nonnegative. We define the \emph{stable subspace}
\begin{align*}
S_{\omega} := F_{\mu_{j_{0}}}(\omega).
\end{align*}
By the \emph{unstable subspace} we mean
\begin{align*}
U_{\omega} := \oplus_{1\leqslant i < j_{0}} H^{i}_{\omega}.
\end{align*}
Note that $\dim[E_{\omega}/S_{\omega}] = \dim[ U_{\omega}] =: k<\infty$ for every $\omega \in \tilde{\Omega}$. 
\end{definition}

\begin{lemma}\label{lemma:prop_F}
 For $\omega \in \tilde{\Omega}$ and $\epsilon \in (0,-\mu_{j_0})$, set 
 \begin{align*}
  F(\omega) := \sup_{p\geqslant 0}\exp[-p(\mu_{j_{0}} + \epsilon)] \Vert\psi^p_\omega |_{S_{\omega}}\Vert.
 \end{align*}
 Then
 \begin{align}\label{BB}
  \lim_{n\rightarrow\infty}\frac{1}{n}\log^{+}\big{[}F(\theta^{n}\omega){]} 
  %=\lim_{n\rightarrow\infty}\frac{1}{n}\log^{+}\big{[}F(\sigma^{n}\omega){]} 
  = 0.
 \end{align}

\end{lemma}

\begin{proof}
 Follows from \eqref{eqn:growth_slow_subspace}. %and Lemma \ref{inve_lim}. \label{lemma:prop_F}
\end{proof}

\begin{lemma}
 Let $ \omega\in\tilde{\Omega} $, $ U_{\omega} = \langle\xi^{t}_{\omega}\rangle_{1\leqslant t\leqslant k} $ and $n, p \geq 0$. Then
 \begin{align}\label{inverse-bound}
 \Vert[\psi^{n}_{\theta^{p}\omega}]^{-1}\Vert_{L[{U_{\theta^{n+p}\omega}},U_{\theta^{p}\omega}]}\leqslant\sum_{1\leqslant t\leqslant k}\frac{\Vert\psi^{p}_{\omega}(\xi^{t}_{\omega})\Vert}{\Vert\psi^{n+p}_{\omega}(\xi^{t}_{\omega})\Vert}\times\frac{\Vert \psi_{\omega}^{n+p}(\xi^{t}_{\omega})\Vert}{d\big{(}\psi_{\omega}^{n+p}(\xi^{t}_{\omega}), \langle \psi_{\omega}^{n+p}(\xi^{t^{\prime}}_{\omega}) \rangle_{t^{\prime}\neq t}\big{)}}
 \end{align}
 and 
 \begin{align}\label{inn_BB}
 \Vert [\psi^{p}_{\sigma^{n}\omega}]^{-1}\Vert_{L[{U_{\sigma^{n-p}\omega}},U_{\sigma^{n}\omega}]}\leqslant \sum_{1\leqslant t\leqslant k}\frac{\Vert (\psi^{n}_{\sigma^{n}\omega})^{-1}(\xi^{t}_{\omega})\Vert}{ \Vert (\psi^{n-p}_{\sigma^{n-p}\omega})^{-1}(\xi^{t}_{\omega})\Vert}\times\frac{ \Vert (\psi^{n-p}_{\sigma^{n-p}\omega})^{-1}(\xi^{t}_{\omega})\Vert}{d\big{(}(\psi^{n-p}_{\sigma^{n-p}\omega})^{-1}(\xi^{t}_{\omega}),\langle (\psi^{n-p}_{\sigma^{n-p}(\omega)})^{-1}(\xi^{t^{\prime}}_{\omega})\rangle_{t^{\prime}\neq t}\big{)}}.
 \end{align}
 
\end{lemma}

 \begin{proof}
Choose $ u\in U_{\theta^{p}\omega} $ and assume that $u = \sum_{1\leqslant t\leqslant k}u^{t}\frac{\psi^{p}_{\omega}(\xi^{t}_{\omega})}{\Vert \psi^{p}_{\omega}(\xi^{t}_{\omega})\Vert} $. Then
\begin{align}\label{inverse}
\frac{\vert u^{t}\vert}{\Vert u\Vert}\leqslant\frac{\Vert \psi_{\omega}^{p}(\xi^{t}_{\omega})\Vert}{d\big{(}\psi_{\omega}^{p}(\xi^{t}_{\omega}),\langle\psi_{\omega}^{p}(\xi^{t^{\prime}}_{\omega})\rangle_{t^{\prime}\neq t}\big{)}}.
\end{align}
From $ \psi^{n}_{\theta^{p}\omega}u = \sum_{1\leqslant t\leqslant k}u^{t} \frac{\Vert\psi^{n+p}_{\omega}(\xi^{t}_{\omega})\Vert}{\Vert\psi^{p}_{\omega}(\xi^{t}_{\omega})\Vert }\frac{\psi^{n+p}_{\omega}(\xi^{t}_{\omega})}{\Vert \psi^{n+p}_{\omega}(\xi^{t}_{\omega})\Vert}$ and \eqref{inverse},
\begin{align*} 
\frac{\vert u^{t}\vert}{\Vert \psi^{n}_{\theta^{p}\omega}u\Vert}\leqslant\frac{\Vert\psi^{p}_{\omega}(\xi^{t}_{\omega})\Vert}{\Vert\psi^{n+p}_{\omega}(\xi^{t}_{\omega})\Vert} \times \frac{\Vert \psi_{\omega}^{n+p}(\xi^{t}_{\omega})\Vert}{d\big{(}\psi_{\omega}^{n+p}(\xi^{t}_{\omega}) ,\langle\psi_{\omega}^{n+p}(\xi^{t^{\prime}}_{\omega})\rangle_{t^{\prime}\neq t}\big{)}}
\end{align*}
and \eqref{inverse-bound} follows. The estimate \eqref{inn_BB} is proven similarly.

 \end{proof}

 \begin{definition}
 For $\omega \in \Omega$ set $ \Sigma_\omega:=\prod_{j\geqslant 0}E_{\theta^{j}\omega}  $. For $ \upsilon>0 $ we define 
  \begin{align*}
 \Sigma^{\upsilon}_{\omega}:=\bigg{\lbrace} \Gamma\in\Sigma_{\omega} :\Vert\Gamma\Vert=\sup_{j\geqslant 0}\big{[}\Vert\Pi_{\omega}^{j}\Gamma\Vert\exp(\upsilon j)\big{]}<\infty \bigg{\rbrace}
 \end{align*}
  where $ \Pi_{\omega}^{j}:\prod_{i\geqslant 0} E_{\theta^{i}\omega}\rightarrow E_{\theta^{j}\omega} $ denotes the projection map. 
 \end{definition}
 
 One can check that $ \Sigma_{\omega}^{\upsilon} $ is a Banach space.
 
 \begin{lemma}\label{func}
  Let $\omega \in {\Omega}$ and $0 < \upsilon < - \mu_{j_0}$. Define
 \begin{align*} 
  P_{\omega} : E_{\omega} &\to E_{\theta\omega }\\
  \xi_{\omega} &\mapsto \varphi^{1}_{\omega} (Y_{\omega}+\xi_{\omega})-\varphi^{1}_{\omega}(Y_{\omega})-\psi^{1}_{\omega}(\xi_{\omega}).
 \end{align*}
 Let $\rho \colon \Omega \to \R^+$ be a random variable with the property that
 \begin{align*}
   \liminf_{n\rightarrow \infty}\frac{1}{n}\log \rho(\theta^{n}\omega) \geq 0
 \end{align*}
 almost surely. Assume that for $ \Vert \xi_{\omega}\Vert , \Vert\tilde{\xi}_{\omega}\Vert <\rho(\omega) $,
  \begin{align}\label{eqn:diff_bound_P}
 \Vert P_{\omega} (\xi_{\omega}) - P_{\omega} (\tilde{\xi}_{\omega})\Vert\leqslant \Vert\xi_{\omega}-\tilde{\xi}_{\omega}\Vert f(\omega)  h{(}\Vert\xi_{\omega}\Vert +\Vert\tilde{\xi}_{\omega}\Vert{)}
 \end{align}
 almost surely where $ f:\Omega\rightarrow \mathbb{R}^{+} $ is a measurable function such that $\lim_{n\rightarrow\infty}\frac{1}{n}\log^{+}f(\theta^{n}\omega) = 0$ almost surely and $ h(x)=x^{r}g(x)$ for some $r > 0$ where $ g:\mathbb{R}\rightarrow\mathbb{R}^{+} $ is an increasing $ C^{1} $ function. Set
  \begin{align}\label{eqn:rhotilde}
    \tilde{\rho}(\omega) := \inf_{n\geqslant 0}\exp(n\upsilon)\rho(\theta^{n}\omega).
  \end{align}
    Then the map
 \begin{align*}
 &I_{_{\omega}} \colon S_{\omega}\times\Sigma_{\omega}^{\upsilon}\cap B(0,\tilde{\rho}(\omega))\rightarrow\Sigma_{\omega}^{\upsilon}, \\
 &{\Pi}^{n}_{\omega}\big{[}I_{\omega}(v_{\omega} ,\Gamma)\big{]} = \begin{cases}
 \psi^{n}_{\omega}(v_{\omega}) + \sum_{0\leqslant j\leqslant n-1}\big{[}\psi^{n-1-j}_{\theta^{1+j}\omega}\circ\Pi_{ S_{\theta^{1+j}\omega}\parallel U_{\theta^{1+j}\omega}}\big{]}P_{\theta^{j}\omega}\big{(}\Pi^{j}_{\omega}[\Gamma]\big{)}    &\text{}\\
 \qquad - \sum_{j\geqslant n}\big{[}[\psi^{j-n+1}_{\theta^{n}\omega}]^{-1}\circ\Pi_{U_{\theta^{1+j}\omega}\parallel S_{\theta^{1+j}\omega}}\big{]}  P_{\theta^{j}\omega}\big{(}\Pi^{j}_{\omega}[\Gamma]\big{)}  &\text{for } n \geq 1,\\
 v_{\omega}- \sum_{j\geqslant 0}\big{[}[\psi^{j+1}_{\omega}]^{-1}\circ\Pi_{U_{\theta^{1+j}\omega}\parallel S_{\theta^{1+j}\omega}}\big{]} P_{\theta^{j}\omega}\big{(}\Pi^{j}_{\omega}[\Gamma]\big{)}  &\text{for } n = 0.
 \end{cases}
 \end{align*}
  is well-defined on a $\theta$-invariant set of full measure $\tilde{\Omega}$.
 \end{lemma}

 \begin{proof}
We collect some estimates first. Let $\epsilon \in (0, -\mu_{j_0})$. From \eqref{pro_invv}, we can find a random variable $ R(\omega)>1 $ such that for $ j\geqslant 0 $,
\begin{align}\label{projection-B}
\Vert \Pi_{U_{\theta^{j}\omega}\parallel S_{\theta^{j}\omega}}\Vert \leqslant R(\omega)\exp(\epsilon j) \ ,\ \ \ \ \  \Vert \Pi_{ S_{\theta^{j}\omega}\parallel U_{\theta^{j}\omega}}\Vert  \leqslant R(\omega)\exp(\epsilon j).
\end{align}
Also from \eqref{BB}, for $ n,p\geqslant 0 $,
\begin{align}\label{VBB}
\Vert\psi^{p}_{\theta^{n}\omega}|_{S_{\theta^{n}\omega}}\Vert\leqslant R(\omega)\exp\big{(}p\mu_{j_{0}}+\epsilon(n+p)\big{)}.
\end{align}
In addition, from \eqref{eqn:dist_forw} and \eqref{inverse-bound} for $ n, p \geqslant 0 $,
\begin{align}\label{inverse-c}
 \Vert[\psi^{n}_{\theta^{p}\omega}]^{-1}\Vert_{L[{U_{\theta^{n+p} \omega}},U_{\theta^{p}\omega}]}\leqslant R(\omega) \exp\big{(}\epsilon(n+p)\big{)} \exp(-n\mu_{j_{0}-1}).
\end{align}
From our assumptions,
\begin{align*}
\big{\Vert}P_{\theta^{j}\omega}\big{(}\Pi^{j}_{\omega}[\Gamma]\big{)}\big{\Vert}\leqslant\big{\Vert} \Pi^{j}_{\omega}[\Gamma]\big{\Vert}^{1+r}\big{[}f(\theta^{j}\omega) g(\Vert \Pi^{j}_{\omega}[\Gamma]\Vert)\big{]}.
\end{align*}
 So for $ j\geqslant 0 $ and a random variable $ \tilde{R}({\omega})>1 $,
 \begin{align}\label{p-}
\big{\Vert}P_{\theta^{j}\omega}\big{(}\Pi^{j}_{\omega}[\Gamma]\big{)}\big{\Vert}\leqslant \tilde{R}({\omega})\big{\Vert} \Pi^{j}_{\omega}[\Gamma]\big{\Vert}^{1+r}g(\Vert \Pi^{j}_{\omega}[\Gamma]\Vert) \exp(\epsilon j).
 \end{align}
Now from \eqref{projection-B}, \eqref{VBB}, \eqref{inverse-c} and \eqref{p-}, we obtain
 \begin{align*}
 &\big{\Vert}{\Pi}^{n}_{\omega}\big{[}I_{\omega}(v_{\omega} ,\Gamma)\big{]}\big{\Vert}\leqslant R(\omega)\bigg{[}\exp((\mu_{j_{0}}+\epsilon)n)\Vert v_{\omega}\Vert + \\&\sum_{0\leqslant j\leqslant n-1}R(\omega) \tilde{R}(\omega)\exp\big{(}\epsilon n+2\epsilon(1+j)+(n-1-j)\mu_{j_{0}} \big{)}\Vert\Pi^{j}_{\omega}(\Gamma)\Vert^{1+r}g({\Vert} \Pi^{j}_{\omega}[\Gamma]\Vert) + \\&\sum_{j\geqslant n}R(\omega)\tilde{R}(\omega)\exp\big{(}3\epsilon(1+j)-(j-n+1)\mu_{j_{0}-1}\big{)}\Vert\Pi^{j}_{\omega}(\Gamma)\Vert^{1+r}g({\Vert} \Pi^{j}_{\omega}[\Gamma]\Vert)\bigg{]}.
 \end{align*}
Since $ g $ is increasing,
 \begin{align*}
&\big{\Vert}{\Pi}^{n}_{\omega}\big{[}I_{\omega}(v_{\omega} ,\Gamma)\big{]}\big{\Vert} \leqslant R(\omega)\bigg{[}\exp\big{(}(\mu_{j_{0}}+\epsilon)n\big{)}.\Vert v_{\omega}\Vert + \\
&R(\omega)\tilde{R}(\omega)\Vert\Gamma\Vert^{1+r}_{\Sigma^{\upsilon}_{\omega}}g(\Vert\Gamma\Vert_{\Sigma^{\upsilon}_{\omega}})\exp\big{(}\epsilon n+2\epsilon+(n-1)\mu_{j_{0}}\big{)}\sum_{0\leqslant j\leqslant n-1}\exp\big{(}j\big{(}2\epsilon-\mu_{j_{0}}-(1+r)\upsilon\big{)}\big{)}+\\
&R(\omega)\tilde{R}(\omega)\Vert\Gamma\Vert^{1+r}_{\Sigma^{\upsilon}_{\omega}}g(\Vert\Gamma\Vert_{\Sigma^{\upsilon}_{\omega}})\exp\big{(}3\epsilon+(n-1)\mu_{j_{0}-1}\big{)}\sum_{j\geqslant n}\exp\big{(}j\big{(}3\epsilon-\mu_{j_{0}-1}-(1+r)\upsilon\big{)}\big{)}\bigg{]}.
 \end{align*}
 Since $\mu_{j_{0}-1} \geqslant 0 $ and $ 0<\upsilon<-\mu_{j_{0}} $, we can choose $\epsilon > 0$ smaller if necessary to see that
 \begin{align*}
 \sup_{n\geqslant 0}\bigg{[}\big{\Vert}{\Pi}^{n}_{\omega}\big{[}I_{\omega}(v_{\omega} ,\Gamma)\big{]}\big{\Vert}\exp(\upsilon n)\bigg{]}<\infty.
 \end{align*}
 As a result, $ I_{{\omega}} $ is well-defined .
  \end{proof}

\begin{lemma}\label{fixed}
With the same setting as in Lemma \ref{func}, for $\Gamma\in \Sigma^{\upsilon}_{\omega} \cap B(0,\tilde{\rho}(\omega))$,
\begin{align}
I_{{\omega}}[v_{\omega},\Gamma]=\Gamma \ \ \ \ {\Longleftrightarrow}\ \ \ \ \forall j\geqslant 0: \Pi^{j}_{\omega}[\Gamma]=\varphi^{j}_{\omega}(Y_{\omega}+\xi_{\omega})-\varphi^{j}_{\omega}(Y_{\omega}) 
\end{align}
where 
\begin{align}\label{CVB}
\xi_{\omega}=v_{\omega}-\sum_{j\geqslant 0}\big{[}[\psi^{j+1}_{\omega}]^{-1}\circ\Pi_{U_{\theta^{1+j}\omega}\parallel S_{\theta^{1+j}\omega}}\big{]} P_{\theta^{j}\omega}\big{(}\Pi^{j}_{\omega}[\Gamma]\big{)}.
\end{align}
\end{lemma}

\begin{proof}
The strategy of the proof is similar to \cite[Lemma VI.5]{Man83}. Let $ I_{\omega}[v_{\omega},\Gamma]=\Gamma $. Then $ \xi_{\omega}=\Pi^{0}_{\omega}[\Gamma] $ and the claim is shown for $j = 0$. We proceed by induction. Assume that $ \Pi^{n}_{\omega}[\Gamma]=\varphi^{n}_{\omega}(Y_{\omega}+\xi_{\omega})-\varphi^{n}_{\omega}(Y_{\omega}) $. By definition, 
\begin{align*}
&\varphi^{n+1}_{\omega}(Y_{\omega}+\xi_{\omega})-\varphi^{n+1}_{\omega}(Y_{\omega})=\varphi^{1}_{\theta^{n}\omega}\big{(}\varphi^{n}_{\omega}(Y_{\omega}+\xi_{\omega})\big{)}-\varphi^{1}_{\theta^{n}\omega}(Y_{\theta^{n}\omega})=\\&P_{\theta^{n}\omega}\big{(}\varphi^{n}_{\omega}(Y_{\omega}+\xi_{\omega})-Y_{\theta^{n}\omega}\big{)} + \psi^{1}_{\theta^{n}\omega}\big{(}\varphi^{n}_{\omega}(Y_{\omega}+\xi_{\omega})-Y_{\theta^{n}\omega}\big{)} = P_{\theta^{n}\omega}(\Pi^{n}_{\omega}[\Gamma]) + \psi^{1}_{\theta^{n}\omega}\big{(}\Pi^{n}_{\omega}\big{[}I_{\omega}(v_{\omega} ,\Gamma)\big{]}\big{)}.
\end{align*}
Note that for $ j\geqslant n$,
\begin{align*}
\psi^{1}_{\theta^{n}\omega}\circ[\psi^{j-n+1}_{\theta^{n}\omega}]^{-1}=[\psi^{j-n}_{\theta^{n+1}\omega}]^{-1}:U_{\theta^{1+j}\omega}\rightarrow U_{\theta^{1+n}\omega}.
\end{align*}
By definition 
\begin{align*}
&\psi^{1}_{\theta^{n}\omega}\big{(}   \Pi^n_{\omega} [I_{\omega}(v_{\omega} ,\Gamma)] \big{)}=\psi^{n+1}_{\omega}(v_{\omega})+\sum_{0\leqslant j\leqslant n-1}\big{[}\psi^{n-j}_{\theta^{1+j}\omega}\circ\Pi_{ S_{\theta^{1+j}\omega}\parallel U_{\theta^{1+j}\omega}}\big{]}P_{\theta^{j}\omega}\big{(}\Pi^{j}_{\omega}[\Gamma]\big{)}-\\&\sum_{j\geqslant n}\big{[}[\psi^{j-n}_{\theta^{n}\omega}]^{-1}\circ\Pi_{U_{\theta^{1+j}\omega}\parallel S_{\theta^{1+j}\omega}}\big{]} P_{\theta^{j}\omega}\big{(}\Pi^{j}_{\omega}[\Gamma]\big{)}.
\end{align*}
Consequently, $\Pi^{n+1}_{\omega}[\Gamma]=\varphi^{n+1}_{\omega}(Y_{\omega}+\xi_{\omega})-\varphi^{n+1}_{\omega}(Y_{\omega}) $ which finishes the induction step. \\
Conversely, for $\xi_{\omega}\in E_{\omega}$ and $\Gamma\in \Sigma^{\nu}_{\omega} \cap B(0,\tilde{\rho}(\omega))$, assume that for every $ j\geqslant 0 $, $ \Pi^{j}_{\omega}[\Gamma]=\varphi^{j}_{\omega}(Y_{\omega}+\xi_{\omega})-\varphi^{j}_{\omega}(Y_{\omega}) $. Set
\begin{align*}
v_{\omega}:=\xi_{\omega}+\sum_{j\geqslant 0}\big{[}[\psi^{j+1}_{\omega}]^{-1}\circ\Pi_{U_{\theta^{1+j}\omega}\parallel S_{\theta^{1+j}\omega}}\big{]} P_{\theta^{j}\omega}\big{(}\Pi^{j}_{\omega}[\Gamma]\big{)}.
\end{align*}
Similar to Lemma \ref{func}, we can see that $ v_{\omega} $ is well-defined. Morever,
\begin{align*}
{\Pi}^{n}_{\omega}\big{[}I_{\omega}(v_{\omega} ,\Gamma)\big{]}&=\psi^{n}_{\omega}(\xi_{\omega})+\sum_{0\leqslant j\leqslant n-1}\psi^{n-1-j}_{\theta^{1+j}\omega}P_{\theta^{j}\omega}\big{(}\Pi^{j}_{\omega}[\Gamma]\big{)}\\&=\varphi^{j}_{\omega}(Y_{\omega}+\xi_{\omega})-\varphi^{j}_{\omega}(Y_{\omega})=\Pi^{j}_{\omega}[\Gamma]
\end{align*}
which proves the claim.
\end{proof}

\begin{lemma}\label{lemma:estimates_I}
 Under the same assumptions as in Lemma \ref{fixed}, set 
\begin{align*}
h_{1}^{\upsilon}(\omega) &:= \sup_{n\geqslant 0}\big{[}\exp(n\upsilon)\Vert\psi^{n}_{\omega}|_{S_{\omega}}\Vert\big{]} \quad \text{and} \\
h_{2}^{\upsilon}(\omega) &:= \sup_{n\geqslant 0}\big{[}\exp(n\upsilon)\sum_{0\leqslant j\leqslant n-1}\exp(-j\upsilon (1+r))f(\theta^{j}\omega)\Vert\psi^{n-j}_{\theta^{j+1}\omega}|_{S_{\theta^{j+1}\omega}}\Vert \Vert\Pi_{S_{\theta^{j+1}\omega}||U_{\theta^{j+1}\omega}}\Vert \\
&\quad + \exp(n\upsilon)\sum_{j\geqslant n}\exp(-j\upsilon (1+r))f(\theta^{j}\omega)\Vert (\psi^{j-n+1}_{\theta^{n}\omega} |_{U_{\theta^{j+1}}} )^{-1}\Vert \Vert\Pi_{U_{\theta^{j+1}\omega}||S_{\theta^{j+1}\omega}}\Vert\big{]}.
\end{align*}
    Then $ h_{1}^{\upsilon} $ and $ h_{2}^{\upsilon}$ are measurable and finite on a $\theta$-invariant set of full measure $\tilde{\Omega}$. In addition,
    \begin{align*}
     \lim_{n \to \infty} \frac{1}{n} \log^+ h^{\upsilon}_1(\theta^n \omega) = \lim_{n \to \infty} \frac{1}{n} \log^+ h^{\upsilon}_2(\theta^n \omega) = 0
    \end{align*}
    for every $\omega \in \tilde{\Omega}$. Furthermore, the estimates
    \begin{align*}
\Vert I_{\omega}(v_{\omega},\Gamma)\Vert &\leqslant h_{1}^{\upsilon}(\omega)\Vert v_{\omega}\Vert + h_{2}^{\upsilon}(\omega)\Vert\Gamma\Vert^{1+r} g(\Vert\Gamma\Vert) \quad \text{and} \\
\Vert I_{\omega}(v_{\omega},\Gamma)-I_{\omega}(v_{\omega},\tilde{\Gamma})\Vert &\leqslant h_{2}^{\upsilon}(\omega)h(\Vert\Gamma\Vert+\Vert\tilde{\Gamma}\Vert)\ \Vert\Gamma-\tilde{\Gamma}\Vert 
\end{align*} 
hold for every $\omega \in \tilde{\Omega}$, $\Gamma, \tilde{\Gamma} \in \Sigma^{\upsilon}_{\omega} \cap B(0,\tilde{\rho}(\omega))$ and $v_{\omega} \in S_{\omega}$.
\end{lemma}

\begin{proof}
 The statements about $h_{1}^{\upsilon}$ and $h_{2}^{\upsilon}$ follow from our assumption on $f$, \eqref{eqn:growth_slow_subspace}, Lemma \ref{lemma:forw_back} and Proposition \ref{In_In_pr}. The claimed estimates follow by definition of $I_{\omega}$.
\end{proof}

Recall that $ h(x)=x^{r}g(x) $. In particular, $h$ is invertible and $h$ and $h^{-1}$ are strictly increasing.

\begin{lemma}\label{lemma:ex_fp_bound}
 Assume that for $v_{\omega} \in S_{\omega}$, 
 \begin{align*}
  \Vert v_{\omega}\Vert\leqslant\frac{1}{2h_{1}^{\upsilon}(\omega)}\min\big{\lbrace}\frac{1}{2} h^{-1}(\frac{1}{2h_{2}^{\upsilon}(\omega)}),\tilde{\rho}(\omega)\big{\rbrace}.
 \end{align*}
 Then the equation
 \begin{align*}
    I_{\omega}(v_{\omega},\Gamma)=\Gamma
 \end{align*}
 admits a uniques solution $\Gamma = \Gamma(v_{\omega})$ and the bound
 \begin{align}\label{eqn:def_H_1}
  \Vert\Gamma (v_{\omega})\Vert\leqslant \min\big{\lbrace}\frac{1}{2}h^{-1}(\frac{1}{2h_{2}^{\upsilon}(\omega)}), \tilde{\rho}(\omega)\big{\rbrace} =: H^{\upsilon}_{1}(\omega)
 \end{align}
    holds true.

\end{lemma}

\begin{proof}
 We can use the estimates provided in Lemma \ref{lemma:estimates_I} to conclude that $I(v_{\omega},\cdot)$ is a contraction on the closed ball with radius $\min\big{\lbrace}\frac{1}{2}h^{-1}(\frac{1}{2h_{2}^{\upsilon}(\omega)}), \tilde{\rho}(\omega)\big{\rbrace}$.
\end{proof}

 Now we can formulate the main theorem about the existence of local stable manifolds.
 
\begin{theorem}\label{stable manifold}
Let $(\Omega,\mathcal{F},\P,\theta)$ be an ergodic measure-preserving dynamical systems and $\varphi$ a Fr\'echet-differentiable cocycle acting on a measurable field of Banach spaces $\{E_{\omega}\}_{\omega \in \Omega}$. Assume that $\varphi$ admits a  stationary solution $Y$ and that the linearized cocycle $\psi$ around $Y$ is compact, satisfies Assumption \ref{MES} and the integrability condition
\begin{align*}
 \log^+ \| \psi_{\omega} \| \in L^1(\omega).
\end{align*}
Moreover, assume that \eqref{eqn:diff_bound_P} holds for $\varphi$ and $\psi$. Let $\mu_{j_0} < 0$ and $S_{\omega}$ be defined as in Definition \ref{stable-dfn}.
For $ 0 < \upsilon < -\mu_{j_{0}} $, $\omega \in {\Omega}$ and $ R^{\upsilon}(\omega) :=\frac{1}{2h_{1}^{\upsilon}(\omega)}\min\big{\lbrace}\frac{1}{2}h^{-1}(\frac{1}{2h_{2}^{\upsilon}(\omega)}),\tilde{\rho}(\omega)\big{\rbrace} $ with $\tilde{\rho}$ defined as in \eqref{eqn:rhotilde}, let
\begin{align}
 S^{\upsilon}_{loc}(\omega) := \big{\lbrace}Y_{\omega}+\Pi^{0}_{\omega}[\Gamma(v_{\omega})], \ \ \Vert v_{\omega}\Vert< R^{\upsilon}(\omega)\big{\rbrace}.
\end{align}
Then there is a $\theta$-invariant set of full measure $\tilde{\Omega}$ on which the following properties are satisfied for every $\omega \in \tilde{\Omega}$:
\begin{itemize}
\item[(i)] There are random variables $ \rho_{1}^{\upsilon}(\omega), \rho_{2}^{\upsilon}(\omega)$, positive and finite on $\tilde{\Omega}$, for which
\begin{align}\label{eqn:rho_temp}
 \liminf_{p \to \infty} \frac{1}{p} \log \rho_i^{\upsilon}(\theta^p \omega) \geq 0, \quad i = 1,2
\end{align}
and such that
\begin{align*}
\big{\lbrace} Z_{\omega} \in E_{\omega}\, :\, \sup_{n\geqslant 0}\exp(n\upsilon)\Vert\varphi^{n}_{\omega}(Z_{\omega})-Y_{\theta^{n}\omega}\Vert &<\rho_{1}^{\upsilon}(\omega)\big{\rbrace}\subseteq S^{\upsilon}_{loc}(\omega)\\&\subseteq \big{\lbrace} Z_{\omega} \in E_{\omega}\, :\, \sup_{n\geqslant 0}\exp(n\upsilon)\Vert\varphi^{n}_{\omega}(Z_{\omega})-Y_{\theta^{n}\omega}\Vert<\rho_{2}^{\upsilon}({\omega})\big{\rbrace}.
\end{align*}
\item[(ii)]$ S^{\upsilon}_{loc}(\omega) $ is an immersed submanifold of $ E_{\omega} $ and
\begin{align*}
 T_{Y_{\omega}}S^{\upsilon}_{loc}(\omega) = S_{\omega}.
\end{align*}
\item[(iii)] For $ n\geqslant N(\omega) $,
\begin{align*}
\varphi^{n}_{\omega}(S^{\upsilon}_{loc}(\omega))\subseteq S^{\upsilon}_{loc}(\theta^{n}\omega).
\end{align*}
\item[(iv)]For $ 0<\upsilon_{1}\leqslant\upsilon_{2}< - \mu_{j_{0}} $,
\begin{align*}
S^{\upsilon_{2}}_{loc}(\omega)\subseteq S^{\upsilon_{1}}_{loc}(\omega).
\end{align*}
Also for $n\geqslant N(\omega) $,
\begin{align*}
\varphi^{n}_{\omega}(S^{\upsilon_{1}}_{loc}(\omega))\subseteq S^{\upsilon_{2}}_{loc}(\theta^{n}(\omega))
\end{align*}
 and consequently for $ Z_{\omega}\in S^{\upsilon}_{loc}(\omega) $,
\begin{align}\label{eqn:contr_char}
    \limsup_{n\rightarrow\infty}\frac{1}{n}\log\Vert\varphi^{n}_{\omega}(Z_{\omega})-Y_{\theta^{n}\omega}\Vert\leqslant  \mu_{j_{0}}.
\end{align}
\item[(v)] 
\begin{align*}
\limsup_{n\rightarrow\infty}\frac{1}{n}\log\bigg{[}\sup\bigg{\lbrace}\frac{\Vert\varphi^{n}_{\omega}(Z_{\omega})-\varphi^{n}_{\omega}(\tilde{Z}_{\omega})\Vert }{\Vert Z_{\omega}-\tilde{Z}_{\omega}\Vert},\ \ Z_{\omega}\neq\tilde{Z}_{\omega},\  Z_{\omega},\tilde{Z}_{\omega}\in S^{\upsilon}_{loc}(\omega) \bigg{\rbrace}\bigg{]}\leqslant \mu_{j_{0}}.
\end{align*}
\end{itemize}
\end{theorem}

\begin{proof}
We start with (i). For the first inclusion, note that we can find a random variable $ \rho_1^{\upsilon}(\omega) $ satisfying
\begin{align}\label{BH}
\liminf_{p\rightarrow\infty}\frac{1}{p}\log \rho_{1}^{\upsilon}(\theta^{p}\omega) \geqslant 0 
\end{align}
and such that whenever $\Vert\Gamma\Vert\leqslant\rho^{\upsilon}_{1}(\omega)$,
\begin{align*}
    \Vert\Gamma\Vert +h_{2}^{\upsilon}(\omega)\Vert\Gamma\Vert^{r+1} g(\Vert\Gamma\Vert)\leqslant \frac{1}{2h_{1}^{\upsilon}(\omega)}\min\big{\lbrace}\frac{1}{2} h^{-1}(\frac{1}{2h_{2}^{\upsilon}(\omega)}),\tilde{\rho}(\omega)\big{\rbrace} =: H^{\upsilon}_{2}(\omega).
\end{align*}
For example, we can define 
\begin{align*}
 \rho_1^{\upsilon}(\omega) :=\min\big{\lbrace}h^{-1}(\frac{1}{h_{2}^{\upsilon}(\omega)}), H^{\upsilon}_{2}(\omega) / 2, H^{\upsilon}_{1}(\omega)\big{\rbrace}
\end{align*}
with $H^{\upsilon}_{1}$ defined as in \eqref{eqn:def_H_1}. Assume that $Z_{\omega} \in E_{\omega}$ has the property that
\begin{align*}
 \sup_{n\geqslant 0}\exp(n\upsilon)\Vert\varphi^{n}_{\omega}(Z_{\omega}) - Y_{\theta^{n}\omega} \Vert < \rho_1^{\upsilon}(\omega).
\end{align*}
Setting
\begin{align*}
 \tilde{v}_{\omega} := Z_{\omega} - Y_{\omega} + \sum_{j\geqslant 0}\big{[}[\psi^{j+1}_{\omega}]^{-1}\circ\Pi_{U_{\theta^{1+j}\omega}\parallel S_{\theta^{1+j}\omega}}\big{]} P_{\theta^{j}\omega}\big{(}\Pi^{j}_{\omega}[\tilde{\Gamma}]\big{)},
\end{align*}
it follows that $\| \tilde{v}_{\omega} \| < R^{\upsilon}(\omega)$. From Lemma \ref{fixed}, we conclude that $I_{\omega}[\tilde{v}_{\omega},\tilde{\Gamma}] = \tilde{\Gamma}$. By uniqueness of the fixed point map, we have $\tilde{\Gamma} = \Gamma(\tilde{v}_{\omega})$, therefore $Z_{\omega} = Y_{\omega} + \Pi^0_{\omega}(\Gamma(\tilde{v}_{\omega})) \in S^{\upsilon}_{loc}(\omega)$. Next, let $Z_{\omega} \in  S^{\upsilon}_{loc}(\omega)$, i.e. $Z_{\omega} = Y_{\omega} + \Pi^0_{\omega}(\Gamma({v}_{\omega}))$ for some $\| v_{\omega} \| < R^{\upsilon}(\omega)$. From Lemma \ref{fixed} and Lemma \ref{lemma:ex_fp_bound},
\begin{align*}
\Vert \Gamma(v_{\omega})\Vert = \sup_{n\geqslant 0}\exp(n\upsilon)\Vert\varphi^{n}_{\omega}(Z_{\omega})-Y_{\theta^{n}\omega} \Vert\leqslant R^{\upsilon}(\omega).
\end{align*}
We can therefore choose $\rho_{2}^{\upsilon}(\omega) = R^{\upsilon}(\omega)$ and the second inclusion is shown.

The second item immediately follows from our definition for $ S^{\upsilon}_{loc}(\omega) $. 

For item (iii), by \eqref{eqn:rho_temp}, we can find $ N(\omega) $ such that for $ n\geqslant N(\omega) $,
\begin{align*}
\exp(-n\upsilon)\rho_{2}^{\upsilon}(\omega)\leqslant \rho_{1}^{\upsilon}(\theta^{n}\omega).
\end{align*}
Now the claim follows from item (i).

For item (iv), note first that $ R^{\upsilon_{2}}(\omega)\leqslant  R^{\upsilon_{1}}(\omega)$. By definition of $ \Gamma^{\upsilon}_{\omega}(v_{\omega}) $, it immediately follows that 
 \begin{align*}
S^{\upsilon_{2}}_{loc}(\omega)\subseteq S^{\upsilon_{1}}_{loc}(\omega).
 \end{align*}
 Now take $ Z_{\omega}\in S^{\upsilon_{1}}_{loc}(\omega) $. From Lemma \ref{lemma:pro_invv} and (i), we can find $ N(\omega) $ such that for $ n\geqslant N(\omega) $,
\begin{align*}
\Vert\Pi_{ S_{\theta^{n}\omega}\parallel U_{\theta^{n}\omega}}\big{(}\varphi^{n}_{\omega}(Z_{\omega})-Y_{\theta^{n}\omega}\big{)}\Vert< R^{\upsilon_{2}}(\theta^{n}\omega).
\end{align*}
We may also assume that $ \exp(-n\upsilon_{1})\rho_{2}^{\upsilon_{1}}(\omega)\leqslant \rho_{1}^{\upsilon_{1}}(\theta^{n}\omega) $ for $n \geq N(\omega)$. For
\begin{align*}
 v_{\theta^{n}\omega} := \Pi_{ S_{\theta^{n}\omega}\parallel U_{\theta^{n}\omega}}\big{(}\varphi^{n}_{\omega}(Z_{\omega})-Y_{\theta^{n}\omega}\big{)}
\end{align*}
let
\begin{align*}
Z_{\theta^{n}\omega} := \Pi^{0}_{\theta^{n}\omega}(\Gamma(v_{\theta^{n}\omega}))+Y_{\theta^{n}\omega}\in S^{\upsilon_{2}}_{loc}(\theta^{n}\omega)\subset S^{\upsilon_{1}}_{loc}(\theta^{n}\omega).
\end{align*}
We claim that $ Z_{\theta^{n}\omega}=\varphi^{n}_{\omega}(Z_{\omega}) $. Since $ Z_{\omega}\in S^{\upsilon_{1}}_{loc}(\omega) $,
\begin{align*}
\sup_{j\geqslant 0}\exp(j\upsilon_{1})\Vert\varphi^{j}_{\theta^{n}\omega}(\varphi^{n}_{\omega}(Z_{\omega}))-Y_{\theta^{j}\theta^{n}\omega}\Vert\leqslant \exp(-n\upsilon_{1})\rho_{2}^{\upsilon_{1}}(\omega)\leqslant\rho_{1}^{\upsilon_{1}}(\theta^{n}\omega).
\end{align*}
So from item (i), $ \varphi^{n}_{\omega}(Z_{\omega})\in S^{\upsilon_{1}}_{loc}(\theta^{n}\omega) $. Remember $ Z_{\theta^{n}\omega}\in S^{\upsilon_{1}}_{loc}(\theta^{n}\omega)\cap S^{\upsilon_{2}}_{loc}(\theta^{n}\omega) $ and
\begin{align*}
\Pi_{S^{\theta^{n}\omega}||U^{\theta^{n}\omega}}(Z_{\theta^{n}\omega}-Y_{\theta^{n}\omega})=\Pi_{S^{\theta^{n}\omega}||U^{\theta^{n}\omega}}(\varphi^{n}_{\omega}(Z_{\omega})-Y_{\theta^{n}\omega}).
\end{align*}
So by uniqueness of the fixed point, we indeed have 
\begin{align*}
    \varphi^{n}_{\omega}(Z_{\omega}) = Z_{\theta^{n}\omega} \in S^{\upsilon_{2}}_{loc}(\theta^{n}\omega).
\end{align*}
To prove \eqref{eqn:contr_char}, let $\upsilon \leq \upsilon_2 < - \mu_0$ and take $Z_{\omega} \in S^{\upsilon}_{loc}(\omega)$. Then we know that for large enough $N$, $\varphi^N_{\omega} (Z_{\omega}) \in S^{\upsilon_2}_{loc}(\theta^N \omega)$, therefore
\begin{align*}
 \sup_{j \geq 0} \exp(j \upsilon_2) \| \varphi^{j+N}_{\omega}(Z_{\omega}) - Y_{\theta^{j+N} \omega} \| < \infty
\end{align*}
and it follows that 
\begin{align*}
  \limsup_{n\rightarrow\infty}\frac{1}{n}\log\Vert\varphi^{n}_{\omega}(Z_{\omega}) - Y_{\theta^{n}\omega}\Vert\leqslant  - \upsilon_2.
\end{align*}
We can choose $\upsilon_2$ arbitrarily close to $- \mu_0$, therefore the claim follows and item (iv) is proved. 

For item (v), first by definition,
\begin{align*}
\Vert \Gamma(v_{\omega})-\Gamma(\tilde{v}_{\omega})\Vert &=\Vert I_{\omega}(v_{\omega},\Gamma(v_{\omega}))-I_{\omega}(\tilde{v}_{\omega},\Gamma(\tilde{v}_{\omega}))\Vert\\&\leqslant \Vert I_{\omega}(v_{\omega},\Gamma(v_{\omega}))-I_{\omega}(\tilde{v}_{\omega},\Gamma({v}_{\omega}))\Vert +\Vert I_{\omega}(\tilde{v}_{\omega},\Gamma(v_{\omega}))-I_{\omega}(\tilde{v}_{\omega},\Gamma(\tilde{v}_{\omega}))\Vert\\&\leqslant h_{1}^{\upsilon}(\omega)\Vert v_{\omega}-\tilde{v}_{\omega}\Vert +\frac{1}{2}\Vert\Gamma(v_{\omega})-\Gamma(\tilde{v}_{\omega})\Vert
\end{align*}
for every $v_{\omega}, \tilde{v}_{\omega} \in S_{\omega}$ with $\|v_{\omega}\|, \| \tilde{v}_{\omega} \| \leq R^{\upsilon}(\omega)$. Consequently,
\begin{align}\label{BBU}
\Vert \Gamma(v_{\omega})-\Gamma(\tilde{v}_{\omega})\Vert\leqslant 2h_{1}^{\upsilon}(\omega)\Vert v_{\omega}-\tilde{v}_{\omega}\Vert.
\end{align}
 Also by definition, cf. \eqref{CVB},
\begin{align*}
\Vert \Pi^{0}_{\omega}(\Gamma(v_{\omega}))-\Pi^{0}_{\omega}(\Gamma(\tilde{v}_{\omega}))\Vert\geqslant \Vert v_{\omega} - \tilde{v}_{\omega}\Vert - h_{2}^{\upsilon}(\omega)\, \Vert\Gamma(v_{\omega})-{\Gamma}_{\omega}(\tilde{v}_{\omega})\Vert\, h(\Vert\Gamma(v_{\omega})\Vert+\Vert{\Gamma}_{\omega}(\tilde{v}_{\omega})\Vert).
\end{align*} 
So from \eqref{BBU}
\begin{align}\label{ZS}
\Vert \Pi^{0}_{\omega}(\Gamma(v_{\omega}))-\Pi^{0}_{\omega}(\Gamma(\tilde{v}_{\omega}))\Vert\geqslant \Vert v_{\omega}-\tilde{v}_{\omega}\Vert \big{[}1 -2h_{1}^{\upsilon}(\omega)h_{2}^{\upsilon}(\omega) h(\Vert\Gamma(v_{\omega})\Vert+\Vert{\Gamma}_{\omega}(\tilde{v}_{\omega})\Vert)\big{]}.
\end{align}
First assume that
\begin{align*}
\max\lbrace\Vert\Gamma(v_{\omega}), \Gamma(\tilde{v}_{\omega})\Vert\rbrace\leqslant\frac{1}{2}h^{-1}(\frac{1}{4h_{1}^{\upsilon}(\omega)h_{2}^{\upsilon}(\omega)}).
\end{align*}
Then from \eqref{BBU} and \eqref{ZS},
\begin{align}\label{Tt}
\frac{\Vert \Gamma(v_{\omega})-\Gamma(\tilde{v}_{\omega})\Vert}{\Vert \Pi^{0}_{\omega}(\Gamma(v_{\omega}))-\Pi^{0}_{\omega}(\Gamma(\tilde{v}_{\omega}))\Vert}\leqslant 4h_{1}^{\upsilon}(\omega).
\end{align}
Thus if $Z_{\omega} = Y_{\omega} + \Pi^0_{\omega}[\Gamma(v_{\omega})]$ and $\tilde{Z}_{\omega} = Y_{\omega} + \Pi^0_{\omega}[\Gamma(v_{\omega})]$, it follows that
\begin{align*}
 \frac{\Vert\varphi^{n}_{\omega}(Z_{\omega})-\varphi^{n}_{\omega}(\tilde{Z}_{\omega})\Vert}{\Vert Z_{\omega}-\tilde{Z}_{\omega}\Vert}\leqslant 4\exp(-n\upsilon) h_{1}^{\upsilon}(\omega)
\end{align*}
for every $n \geq 1$. In the general case, we can use item (i) and that $ h^{-1}(\frac{1}{4h_{1}^{\upsilon}(\omega)h_{2}^{\upsilon}(\omega)}) $ satisfies \eqref{eqn:rho_temp} to see that for some $ N=N(\omega) $,
\begin{align*}
\sup_{j\geqslant 0}\exp(j\upsilon)\Vert\varphi^{j}_{\theta^{N}\omega}(\varphi^{N}_{\omega}(Z_{\omega}))-Y_{\theta^{j}\theta^{N}\omega}\Vert\leqslant \exp(-N\upsilon)\rho^{\upsilon}_{2}(\omega)\leqslant \frac{1}{2}h^{-1}(\frac{1}{4h_{1}^{\upsilon}(\theta^{N}\omega)h_{2}^{\upsilon}(\theta^{N}\omega)}).
\end{align*} 
Consequently, from \eqref{Tt},
\begin{align*}
  \sup_{j\geqslant 0} \frac{\exp(j\upsilon)\Vert\varphi^{j+N}_{\omega}(Z_{\omega})-\varphi^{j+N}_{\omega}(\tilde{Z}_{\omega})\Vert}{\Vert\varphi^{N}_{\omega}(Z_{\omega})-\varphi^{N}_{\omega}(\tilde{Z}_{\omega})\Vert}\leqslant 4h_{1}^{\upsilon}(\theta^{N}\omega)
\end{align*}
and hence for every $ n\geqslant N $,
\begin{align}\label{CR}
\frac{\Vert\varphi^{n}_{\omega}(Z_{\omega})-\varphi^{n}_{\omega}(\tilde{Z}_{\omega})\Vert}{\Vert Z_{\omega}-\tilde{Z}_{\omega}\Vert}\leqslant 4\exp((-n-N)\upsilon) h_{1}^{\upsilon}(\theta^{N}\omega)H^{\upsilon}_{N}(\omega)
\end{align}
where
\begin{align*}
H^{\upsilon}_{N}(\omega)=\sup\bigg{\lbrace}\frac{\Vert\varphi^{N}_{\omega}(Z_{\omega})-\varphi^{N}_{\omega}(\tilde{Z}_{\omega})\Vert }{\Vert Z_{\omega}-\tilde{Z}_{\omega}\Vert},\ \ Z_{\omega}\neq\tilde{Z}_{\omega},\  Z_{\omega},\tilde{Z}_{\omega}\in S^{\upsilon}_{loc}(\omega) \bigg{\rbrace}.
\end{align*}
We claim that $ H^{\upsilon}_{N}(\omega) $ is finite. Indeed, by assumption \eqref{eqn:diff_bound_P},
\begin{align*}
&\Vert\varphi^{N}_{\omega}(Z_{\omega})-\varphi^{N}_{\omega}(\tilde{Z}_{\omega})\Vert\leqslant \Vert\psi^{1}_{\theta^{N-1}\omega}\Vert\  \Vert\varphi^{N-1}_{\omega}(Z_{\omega})-\varphi^{N-1}_{\omega}(\tilde{Z}_{\omega})\Vert \\ &\qquad + f(\theta^{N}\omega) \ \Vert \varphi^{N-1}_{\omega}(Z_{\omega})-\varphi^{N-1}_{\omega}(\tilde{Z}_{\omega})\Vert  h\big{(}\Vert\varphi^{N-1}_{\omega}(Z_{\omega})-Y_{\theta^{N-1}\omega}\Vert +\Vert\varphi^{N-1}_{\omega}(\tilde{Z}_{\omega})-Y_{\theta^{N-1}\omega}\Vert\big{)}
\end{align*} 
and we can proceed by induction to conclude. Finally, from \eqref{CR} and item (iv), our claim is proved.
\end{proof}

\begin{remark}
Assume that for $ \omega \in \tilde{\Omega} $ the function $ \varphi_{\omega} $  is $ C^{m} $. Then, since
\begin{align*}
I_{\omega}(0,0)=\frac{\partial}{\partial\Gamma}I_{\omega}(0,0)=0,
\end{align*}
we can deduce from the Implicit function theorem that $S^{\upsilon}_{loc}(\omega) $ is locally $ C^{m-1} $.
\end{remark}

\subsection{Unstable manifolds}
We invoke same strategy for proving the existence of unstable manifolds. Since the arguments are very similar, we will only sketch them briefly. In this section, we will assume that the largest Lyapunov exponent is strictly positive, i.e. that $\mu_1 > 0$. 
 \begin{definition}\label{unstable-dfn}
  Set  $k_{0} := \min \lbrace k :\mu_{k} > 0\rbrace $, $\tilde{S}_{\omega} := F_{\mu_{k_{0}+1}}(\omega) $ and $ \tilde{U}_{\omega}=\oplus_{1\leqslant i\leqslant k_{0}}H^{i}_{\omega} $ for $\omega \in \tilde{\Omega}$. For $ \tilde{\Sigma}_\omega := \prod_{j\geqslant 0}E_{\sigma^{j}\omega}  $ and $ \upsilon>0 $, we define the Banach space
  \begin{align*}
 \tilde{\Sigma}^{\upsilon}_{\omega} := \bigg{\lbrace} {\Gamma}\in \tilde{\Sigma}_{\omega}\ :\ \Vert{\Gamma}\Vert = \sup_{k\geqslant 0}\big{[}\Vert\tilde{\Pi}_{\omega}^{k}{\Gamma}\Vert\exp( k\upsilon)\big{]}<\infty \bigg{\rbrace}
 \end{align*}
 where $ \tilde{\Pi}_{\omega}^{k}:\prod_{i\geqslant 0} E_{\sigma^{i}\omega}\rightarrow E_{\sigma^{k}\omega} $ is the  projection map. Similar to last section, we also set
\begin{align*}
\tilde{h}_{1}^{\upsilon}(\omega) &:= \sup_{n\geqslant 0}\big{[}\exp(n\upsilon)\Vert (\psi^{n}_{\sigma^{n}\omega} |_{\tilde{U}_{\omega}} )^{-1} \Vert\big{]} \quad \text{and} \\
\tilde{h}_{2}^{\upsilon}(\omega) &:= \sup_{n\geqslant 0}\big{[}\exp(n\upsilon) \sum_{0\leqslant k\leqslant n-1} \exp\big{(}-\upsilon(n-k) (1+r) \big{)} f(\sigma^{n-k}\omega) \Vert (\psi^{k+1}_{\sigma^{n}\omega} |_{\tilde{U}_{\sigma^{n-1-k}\omega}} )^{-1}\Vert \\
&\qquad \times \Vert\Pi_{ \tilde{U}_{\sigma^{n-1-k}\omega}\parallel \tilde{S}_{\sigma^{n-1-k}\omega}}\Vert \\
&\quad + \exp(n\upsilon)\sum_{k\geqslant n}\exp(-\upsilon (k+1)(1+r))f(\sigma^{k+1}\omega)\Vert \psi^{k-n}_{\sigma^{k}\omega}|_{\tilde{S}_{\sigma^{k}\omega}}\Vert \Vert\Pi_{\tilde{S}_{\sigma^{k}\omega}||\tilde{U}_{\sigma^{k}\omega}}\Vert\big{]}.
\end{align*}
  \end{definition}

  \begin{lemma}
   Let $\omega \in \Omega$, $ 0 < \upsilon < \mu_{k_{0}}$ and assume that $\rho \colon \Omega \to \R^+$ satisfies 
   \begin{align}\label{eqn:rho_sigma_prop}
    \liminf_{n \to \infty} \frac{1}{n} \log \rho(\sigma^n \omega) \geq 0
   \end{align}
   almost surely. Define $P$ as in Lemma \ref{func} and assume that  \eqref{eqn:diff_bound_P} holds for a random variable $f \colon \Omega \to \R^+$ which satisfies $\lim_{n \to \infty} f(\sigma^n \omega) = 0$ almost surely. Set 
   \begin{align}\label{eqn:def_tilde_rho_2}
    \tilde{\rho}(\omega) := \inf_{n \geq 0} \exp(n \upsilon) \rho(\sigma^n \omega).
   \end{align}
   Then the map
   \begin{align*}
 &\tilde{I}_{_{\omega}} :\tilde{U}_{\omega}\times\tilde{\Sigma}_{\omega}^{\upsilon}\cap B(0,\tilde{\rho}(\omega))\rightarrow\tilde{\Sigma}_{\omega}^{\upsilon}, \\
 &{\tilde{\Pi}}^{n}_{\omega}\big{[} \tilde{I} _{\omega}(u_{\omega} ,\Gamma)\big{]} = \begin{cases}
 [\psi^{n}_{\sigma^{n}\omega}]^{-1}(u_{\omega}) &\text{}\\
 \qquad -\sum_{0\leqslant k\leqslant n-1}\big{[}[\psi^{k+1}_{\sigma^{n}\omega}]^{-1}\circ\Pi_{ \tilde{U}_{\sigma^{n-1-k}\omega}\parallel \tilde{S}_{\sigma^{n-1-k}\omega}}\big{]}P_{\sigma^{n-k}\omega}\big{(}\tilde{\Pi}^{n-k}_{\omega}[\Gamma]\big{)}   &\text{}\\
 \qquad + \sum_{k\geqslant n}\big{[}\psi^{k-n}_{\sigma^{k}\omega}\circ\Pi_{\tilde{S}_{\sigma^{k}\omega}\parallel \tilde{U}_{\sigma^{k}\omega}}\big{]} P_{\sigma^{k+1}\omega}\big{(}\tilde{\Pi}^{k+1}_{\omega}[\Gamma]\big{)}  &\text{for } n \geq 1,\\
 u_{\omega}+ \sum_{k\geqslant 0}\big{[}\psi^{k}_{\sigma^{k}\omega}\circ\Pi_{\tilde{S}_{\sigma^{k}\omega}\parallel \tilde{U}_{\sigma^{k}\omega}}\big{]} P_{\sigma^{k+1}\omega}\big{(}\tilde{\Pi}^{k+1}_{\omega}[\Gamma]\big{)}  &\text{for } n = 0.
 \end{cases}
 \end{align*}
  is well-defined on a $\theta$-invariant set of full measure $\tilde{\Omega}$.

  \end{lemma}

  \begin{proof}
   We can use Lemma \ref{inve_lim} to obtain a version of Lemma \ref{lemma:prop_F} where we replace $\theta$ by $\sigma$. The rest of the proof is similar to Lemma \ref{func}.
  \end{proof}

\begin{lemma}
 For $ 0 < \upsilon < \mu_{k_{0}}$, $\omega \in \tilde{\Omega}$ and $\Gamma\in \Sigma^{\upsilon}_{\omega} \cap B(0,\tilde{\rho}(\omega))$,
\begin{align}\label{UY}
\tilde{I}_{\omega}(u_{\omega},{\Gamma})={\Gamma} \ \ \ \ {\Longleftrightarrow}\ \ \ \ \forall\ 0 \leq k\leqslant n: \ \ \ \tilde{\Pi}^{n-k}_{\omega}{\Gamma} = \varphi^{k}_{\sigma^{n}\omega}( \tilde{\Pi}^{n}_{\omega}{\Gamma}+Y_{\sigma^{n}\omega})-Y_{\sigma^{n-k}\omega} .
\end{align}
% where 
% \begin{align}\label{CVB}
% \xi_{\omega}=v_{\omega}-\sum_{j\geqslant 0}\big{[}[\psi^{j+1}_{\omega}]^{-1}\circ\Pi_{U_{\theta^{1+j}\omega}\parallel S_{\theta^{1+j}\omega}}\big{]} P_{\theta^{j}\omega}\big{(}\Pi^{j}_{\omega}[\Gamma]\big{)}.
% \end{align}
\end{lemma}

\begin{proof}
 Similar to Lemma \ref{fixed}.
\end{proof}

\begin{lemma}\label{lemma:estimates_II}
  For $ 0 < \upsilon < \mu_{k_{0}}$, $\tilde{h}_{1}^{\upsilon} $ and $\tilde{h}_{2}^{\upsilon}$ are measurable and finite on a $\theta$-invariant set of full measure $\tilde{\Omega}$. Moreover,
    \begin{align}\label{BX}
\lim_{p\rightarrow\infty}\frac{1}{p}\log^{+}\tilde{h}^{\upsilon}_{1}(\sigma^{p}\omega)=\lim_{p\rightarrow\infty}\frac{1}{p}\log^{+}\tilde{h}^{\upsilon}_{2}(\sigma^{p}\omega)=0
\end{align}
and
\begin{align*}
&\Vert \tilde{I}_{\omega}(u_{\omega},{\Gamma})\Vert\leqslant \tilde{h}_{1}^{\upsilon}(\omega)\Vert u_{\omega}\Vert +\tilde{h}_{2}^{\upsilon}(\omega)\Vert{\Gamma}\Vert^{r+1} g(\Vert{\Gamma}\Vert)\\
&\Vert \tilde{I}_{\omega}(u_{\omega},\Gamma) - \tilde{I}_{\omega}(u_{\omega},\tilde{\Gamma})\Vert\leqslant \tilde{h}_{2}^{\upsilon}(\omega)h(\Vert\Gamma\Vert+\Vert\tilde{\Gamma}\Vert)\ \Vert\Gamma-\tilde{\Gamma}\Vert 
\end{align*}
hold for every $\omega \in \tilde{\Omega}$, $\Gamma, \tilde{\Gamma} \in \tilde{\Sigma}^{\upsilon}_{\omega} \cap B(0,\tilde{\rho}(\omega))$ and $u_{\omega} \in \tilde{U}_{\omega}$. 
\end{lemma}

\begin{proof}
 As in Lemma \ref{lemma:estimates_I}.
\end{proof}

\begin{lemma}\label{lemma:ex_fp_bound_2}
 Assume that for $u_{\omega} \in \tilde{U}_{\omega}$, 
 \begin{align*}
  \Vert u_{\omega}\Vert\leqslant \frac{1}{2\tilde{h}_{1}^{\upsilon}(\omega)}\min\big{\lbrace}\frac{1}{2}h^{-1}(\frac{1}{2\tilde{h}_{2}^{\upsilon}(\omega)}),\tilde{\rho}(\omega)\big{\rbrace}.
 \end{align*}
 Then the equation
 \begin{align*}
    \tilde{I}_{\omega}(u_{\omega},\Gamma) = \Gamma
 \end{align*}
 admits a uniques solution $\Gamma = \Gamma(u_{\omega})$ and the bound
 \begin{align*}
  \Vert\Gamma (u_{\omega})\Vert \leqslant \min\big{\lbrace}\frac{1}{2}h^{-1}(\frac{1}{2\tilde{h}_{2}^{\upsilon}(\omega)}), \tilde{\rho}(\omega)\big{\rbrace} % =: H^{\upsilon}_{1}(\omega)
 \end{align*}
    holds true.

\end{lemma}

\begin{proof}
 We can show that $\tilde{I}(u_{\omega},\cdot)$ is a contraction using Lemma \ref{lemma:estimates_II}.
\end{proof}

Finally we can formulate our main results about the existence of local unstable manifolds.
\begin{theorem}\label{unstable manifold}
Let $(\Omega,\mathcal{F},\P,\theta)$ be an ergodic measure-preserving dynamical systems, $\sigma := \theta^{-1}$ and $\varphi$ a Fr\'echet-differentiable cocycle acting on a measurable field of Banach spaces $\{E_{\omega}\}_{\omega \in \Omega}$. Assume that $\varphi$ admits a stationary solution $Y$ and that the linearized cocycle $\psi$ around $Y$ is compact, satisfies Assumption \ref{MES} and the integrability condition
\begin{align*}
 \log^+ \| \psi_{\omega} \| \in L^1(\omega).
\end{align*}
Moreover, assume that \eqref{eqn:diff_bound_P} holds for $\varphi$ and $\psi$ and a random variable $\rho \colon \Omega \to \R^+$ satisfying \eqref{eqn:rho_sigma_prop}. Assume that $\mu_1 > 0 $ and let $\mu_{k_0} > 0$ and $\tilde{U}_{\omega}$ be defined as in Definition \ref{unstable-dfn}.
For $ 0 < \upsilon < \mu_{k_{0}} $, $\omega \in {\Omega}$ and $ R^{\upsilon}(\omega) :=\frac{1}{2 \tilde{h}_{1}^{\upsilon}(\omega)} \min\big{\lbrace}\frac{1}{2}h^{-1}(\frac{1}{2\tilde{h}_{2}^{\upsilon}(\omega)}),\tilde{\rho}(\omega)\big{\rbrace} $ with $\tilde{\rho}$ defined as in \eqref{eqn:def_tilde_rho_2}, let
\begin{align}
 U^{\upsilon}_{loc}(\omega) := \big{\lbrace}Y_{\omega}+\tilde{\Pi}^{0}_{\omega}[\Gamma(u_{\omega})], \ \ \Vert u_{\omega}\Vert< \tilde{R}^{\upsilon}(\omega)\big{\rbrace}.
\end{align}
Then there is a $\theta$-invariant set of full measure $\tilde{\Omega}$ on which the following properties are satisfied for every $\omega \in \tilde{\Omega}$:

\begin{itemize}
\item[(i)] There are random variables $ \tilde{\rho}_{1}^{\upsilon}(\omega), \tilde{\rho}_{2}^{\upsilon}(\omega)$, positive and finite on $\tilde{\Omega}$, for which
\begin{align*}
 \liminf_{p \to \infty} \frac{1}{p} \log \tilde{\rho}_i^{\upsilon}(\sigma^p \omega) \geq 0, \quad i = 1,2
\end{align*}
and such that
\begin{align*}
&\bigg{\lbrace} Z_{\omega} \in E_{\omega}\, :\, \exists \lbrace Z_{\sigma^{n}\omega}\rbrace_{n\geqslant 1} \text{ s.t. } \varphi^{m}_{\sigma^{n}\omega}(Z_{\sigma^{n}\omega}) = Z_{\sigma^{n-m}\omega} \text{ for all } 0 \leq m \leq n \text{ and }\\ 
&\quad \sup_{n\geqslant 0}\exp(n\upsilon)\Vert Z_{\sigma^{n}\omega} - Y_{\sigma^{n}\omega} \Vert <\tilde{\rho}_{1}^{\upsilon}(\omega)\bigg{\rbrace} \subseteq U^{\upsilon}_{loc}(\omega) \subseteq \bigg{\lbrace}  Z_{\omega} \in E_{\omega}\, :\, \exists \lbrace Z_{\sigma^{n}\omega}\rbrace_{n\geqslant 1} \text{ s.t. } \\
&\qquad \varphi^{m}_{\sigma^{n}\omega}(Z_{\sigma^{n}\omega} ) = Z_{\sigma^{n-m}\omega} \text{ for all } 0 \leq m \leq n \text{ and }  \sup_{n\geqslant 0}\exp(n\upsilon)\Vert Z_{\sigma^{n}\omega} - Y_{\sigma^{n}\omega}\Vert <\tilde{\rho}_{2}^{\upsilon}(\omega)\bigg{\rbrace}.
\end{align*}
\item[(ii)]$ U^{\upsilon}_{loc}(\omega) $ is an immersed submanifold of $E_{\omega} $ and
\begin{align*}
 T_{Y_{\omega}}U^{\upsilon}_{loc}(\omega) = \tilde{U}_{\omega}.
\end{align*}
\item[(iii)] For $ n\geqslant N(\omega) $,
\begin{align*}
U^{\upsilon}_{loc}(\omega)\subseteq \varphi^{n}_{\sigma^{n}\omega}(U^{\upsilon}_{loc}(\sigma^{n}\omega)).
\end{align*}
\item[(iv)] For $ 0<\upsilon_{1} \leqslant \upsilon_{2} < \mu_{k_{0}} $,
\begin{align*}
U^{\upsilon_{2}}_{loc}(\omega)\subseteq U^{\upsilon_{1}}_{loc}(\omega).
\end{align*}
Also for $n\geqslant N(\omega) $,
\begin{align*}
U^{\upsilon_{1}}_{loc}(\omega)\subseteq \varphi^{n}_{\sigma^{n}\omega}(U^{\upsilon_{2}}_{loc}(\sigma^{n}(\omega))
\end{align*}
and consequently for $ Z_{\omega}\in U^{\upsilon}_{loc}(\omega) $, % for $ (\varphi^{n}_{\sigma^{n}\omega})^{-1}(Z_{\omega}):=Z_{\sigma^{n}\omega} $
\begin{align*}
\limsup_{n\rightarrow\infty}\frac{1}{n}\log\Vert Z_{\sigma^{n}\omega} - Y_{\sigma^{n}\omega}\Vert\leqslant -\mu_{k_{0}}.
\end{align*}
\item[(v)] 
\begin{align*}
\limsup_{n\rightarrow\infty}\frac{1}{n}\log\bigg{[}\sup\bigg{\lbrace}\frac{\Vert Z_{\sigma^{n}\omega}- \tilde{Z}_{\sigma^{n}\omega}\Vert }{\Vert Z_{\omega}-\tilde{Z}_{\omega}\Vert},\ \ Z_{\omega}\neq\tilde{Z}_{\omega},\  Z_{\omega},\tilde{Z}_{\omega}\in U^{\upsilon}_{loc}(\omega) \bigg{\rbrace}\bigg{]}\leqslant -\mu_{k_{0}}.
\end{align*}
\end{itemize}
\end{theorem}

\begin{proof}
    One uses the same arguments as in the proof of Theorem \ref{stable manifold}.
\end{proof}

\begin{remark}
 \begin{itemize}
  \item[(i)] As in the stable case, if $ \varphi_{\omega} $  is $ C^{m} $ for every $ \omega \in \tilde{\Omega} $, one can deduce that ${U}^{\upsilon}_{loc}(\omega) $ is locally $ C^{m-1} $.
  \item[(ii)] In the \emph{hyperbolic} case, i.e. if all Lyapunov exponents are non-zero, if the assumptions of Theorem \ref{stable manifold} and \ref{unstable manifold} are satisfied, we have $S_{\omega} = \tilde{S}_{\omega}$ and $U_{\omega} = \tilde{U}_{\omega}$. In particular, the submanifolds $S^{\upsilon}_{loc}(\omega) $ and $U^{\upsilon}_{loc}(\omega) $ are \emph{transversal}, i.e.
  \begin{align*}
   E_{\omega} = T_{Y_{\omega}} U^{\upsilon}_{loc}(\omega) \oplus T_{Y_{\omega}} S^{\upsilon}_{loc}(\omega).
  \end{align*}

 \end{itemize}

\end{remark}

 \subsection*{Acknowledgements}
\label{sec:acknowledgements}

MGV acknowledges a scholarship from the Berlin Mathematical School (BMS). SR acknowledges financial support by the DFG via Research Unit FOR 2402.

\bibliographystyle{alpha}
\bibliography{refs}

\def\cprime{$'$} \def\cprime{$'$}
\begin{thebibliography}{DPDT11}

\bibitem[Arn98]{Arn98}
Ludwig Arnold.
\newblock {\em Random dynamical systems}.
\newblock Springer Monographs in Mathematics. Springer-Verlag, Berlin, 1998.

\bibitem[Blu16]{Blu16}
Alex Blumenthal.
\newblock A volume-based approach to the multiplicative ergodic theorem on
  {B}anach spaces.
\newblock {\em Discrete Contin. Dyn. Syst.}, 36(5):2377--2403, 2016.

\bibitem[CGAS07]{CRS07}
T.~Caraballo, M.~J. Garrido-Atienza, and B.~Schmalfuss.
\newblock Existence of exponentially attracting stationary solutions for delay
  evolution equations.
\newblock {\em Discrete Contin. Dyn. Syst.}, 18(2-3):271--293, 2007.

\bibitem[CKS04]{CKS04}
Tom\'{a}s Caraballo, Peter~E. Kloeden, and Bj\"{o}rn Schmalfu\ss.
\newblock Exponentially stable stationary solutions for stochastic evolution
  equations and their perturbation.
\newblock {\em Appl. Math. Optim.}, 50(3):183--207, 2004.

\bibitem[CPT13]{CPT13}
Monica Conti, Vittorino Pata, and Roger Temam.
\newblock Attractors for processes on time-dependent spaces. {A}pplications to
  wave equations.
\newblock {\em J. Differential Equations}, 255(6):1254--1277, 2013.

\bibitem[Doa09]{DoaPhd09}
Thai~Son Doan.
\newblock {\em Lyapunov Exponents for Random Dynamical Systems}.
\newblock PhD thesis, Technische Universit\" at Dresden, 2009.

\bibitem[DPDT11]{DPDT11}
Francesco Di~Plinio, Gregory~S. Duane, and Roger Temam.
\newblock Time-dependent attractor for the oscillon equation.
\newblock {\em Discrete Contin. Dyn. Syst.}, 29(1):141--167, 2011.

\bibitem[FH14]{FH14}
Peter~K. Friz and Martin Hairer.
\newblock {\em A Course on {R}ough {P}aths with an introduction to regularity
  structures}, volume XIV of {\em Universitext}.
\newblock Springer, Berlin, 2014.

\bibitem[FS96]{FS96}
Franco Flandoli and Bj{\"o}rn Schmalfuss.
\newblock Random attractors for the {$3$}{D} stochastic {N}avier-{S}tokes
  equation with multiplicative white noise.
\newblock {\em Stochastics Stochastics Rep.}, 59(1-2):21--45, 1996.

\bibitem[GTQ14]{GTQ14}
Cecilia Gonz\'{a}lez-Tokman and Anthony Quas.
\newblock A semi-invertible operator {O}seledets theorem.
\newblock {\em Ergodic Theory Dynam. Systems}, 34(4):1230--1272, 2014.

\bibitem[GTQ15]{GTQ15}
Cecilia Gonz\'{a}lez-Tokman and Anthony Quas.
\newblock A concise proof of the multiplicative ergodic theorem on {B}anach
  spaces.
\newblock {\em J. Mod. Dyn.}, 9:237--255, 2015.

\bibitem[GVRS]{GRS19}
Mazyar Ghani~Varzaneh, Sebastian Riedel, and Micheal Scheutzow.
\newblock A dynamical theory for singular stochastic delay differential
  equations {I}: {L}inear equations and a {M}ultiplicative {E}rgodic {T}heorem
  on fields of {B}anach spaces.
\newblock {\em arXiv:1903.01172v3}, 2019.

\bibitem[Kat95]{Kat95}
Tosio Kato.
\newblock {\em Perturbation theory for linear operators}.
\newblock Classics in Mathematics. Springer-Verlag, Berlin, 1995.
\newblock Reprint of the 1980 edition.

\bibitem[Lio61]{Lio62}
J.-L. Lions.
\newblock {\em \'{E}quations diff\'{e}rentielles op\'{e}rationnelles et
  probl\`emes aux limites}.
\newblock Die Grundlehren der mathematischen Wissenschaften, Bd. 111.
  Springer-Verlag, Berlin-G\"{o}ttingen-Heidelberg, 1961.

\bibitem[LL10]{LL10}
Zeng Lian and Kening Lu.
\newblock Lyapunov exponents and invariant manifolds for random dynamical
  systems in a {B}anach space.
\newblock {\em Mem. Amer. Math. Soc.}, 206(967):vi+106, 2010.

\bibitem[LPP95]{LPP95}
Russell Lyons, Robin Pemantle, and Yuval Peres.
\newblock Conceptual proofs of {$L\log L$} criteria for mean behavior of
  branching processes.
\newblock {\em Ann. Probab.}, 23(3):1125--1138, 1995.

\bibitem[Mn83]{Man83}
Ricardo Ma\~{n}\'{e}.
\newblock Lyapounov exponents and stable manifolds for compact transformations.
\newblock In {\em Geometric dynamics ({R}io de {J}aneiro, 1981)}, volume 1007
  of {\em Lecture Notes in Math.}, pages 522--577. Springer, Berlin, 1983.

\bibitem[Ose68]{Ose68}
V.~I. Oseledec.
\newblock A multiplicative ergodic theorem. {C}haracteristic {L}japunov,
  exponents of dynamical systems.
\newblock {\em Trudy Moskov. Mat. Ob\v s\v c.}, 19:179--210, 1968.

\bibitem[Rag79]{Rag79}
M.~S. Raghunathan.
\newblock A proof of {O}seledec's multiplicative ergodic theorem.
\newblock {\em Israel J. Math.}, 32(4):356--362, 1979.

\bibitem[Rue79]{Rue79}
David Ruelle.
\newblock Ergodic theory of differentiable dynamical systems.
\newblock {\em Inst. Hautes \'Etudes Sci. Publ. Math.}, (50):27--58, 1979.

\bibitem[Rue82]{Rue82}
David Ruelle.
\newblock Characteristic exponents and invariant manifolds in {H}ilbert space.
\newblock {\em Ann. of Math. (2)}, 115(2):243--290, 1982.

\bibitem[Thi87]{Thi87}
P.~Thieullen.
\newblock Fibr\'{e}s dynamiques asymptotiquement compacts. {E}xposants de
  {L}yapounov. {E}ntropie. {D}imension.
\newblock {\em Ann. Inst. H. Poincar\'{e} Anal. Non Lin\'{e}aire}, 4(1):49--97,
  1987.

\bibitem[Wal93]{Wal93}
Peter Walters.
\newblock A dynamical proof of the multiplicative ergodic theorem.
\newblock {\em Trans. Amer. Math. Soc.}, 335(1):245--257, 1993.

\bibitem[Woj91]{Woj91}
P.~Wojtaszczyk.
\newblock {\em Banach spaces for analysts}, volume~25 of {\em Cambridge Studies
  in Advanced Mathematics}.
\newblock Cambridge University Press, Cambridge, 1991.

\end{thebibliography}

\end{document}